\documentclass[12pt,reqno]{amsart}
\usepackage{latexsym, amsmath, amscd, amssymb, amsthm,longtable} 
\usepackage[T2A]{fontenc}
 \usepackage[cp1251]{inputenc}
\usepackage[english]{babel}
\newtheorem{theorem}{Theorem}[section]

\newtheoremstyle{myexample}
  {\topsep}
  {\topsep}
  {\normalfont}
  {}
  {\bfseries}
  {.}
  {.5em}
  {}

\theoremstyle{myexample}
\newtheorem{examplenb}{Example}[section]

\newenvironment{example}
  {\begin{examplenb}}
  {\hfill$\triangle$\end{examplenb}}

\newtheorem{lemma}{Lemma}[section]

\usepackage{hyperref}
\usepackage{mathrsfs}

\begin{document}

\noindent
\title[Projective Differential Moment Invariants] {Projective Differential Moment Invariants from the Viewpoint of Classical Invariant Theory}

\author{Leonid Bedratyuk}
\address{Khmelnytskyi National University, Insituts'ka st., 11, Khmelnytskyi, 29016, Ukraine}
\email{leonidbedratyuk@khmnu.edu.ua}
\begin{abstract}
Polynomial first-order projective differential moment
invariants are investigated.
The main result establishes an exact correspondence
between such moment invariants and symmetric polynomial
multipoint relative projective differential invariants
of weight $-1$. This correspondence provides an algebraic
description of each homogeneous component, with respect
to moment degree, of the space of polynomial projective
differential moment invariants and makes it possible to use
previously obtained results on relative projective invariants
for their systematic construction.
It is also proved separately that the passage from the
formal algebra of moments to their integral definitions
does not produce additional algebraic relations among
the constructed invariants.
For low moment degrees, the absence of nontrivial invariants
of degrees $1$ and $2$ is established, and a complete
description of the invariants of degree $3$ is obtained.
The ten cubic differential moments are naturally identified
with the coefficients of a ternary cubic, thereby establishing
a connection between the theory developed here and classical
invariant theory.
In particular, the Aronhold invariants $S$ and $T$ generate
projective differential moment invariants of degrees $4$
and $6$. It is also shown that one of the known moment
invariants of Wang et al. is proportional to the invariant $S$.
It is proved that the algebra of polynomial moment
invariants is not finitely generated.
It is also established that the field of rational
moment invariants coincides with the field of fractions
of the algebra of polynomial moment invariants.
For the explicit constructions, compact multipoint
representations are provided that define the corresponding
moment polynomials without the need to write out their
full expansions in the text.
These representations provide a basis for the algorithmic
construction and software computation of the invariants.

\end{abstract}

\maketitle  
\section{Introduction}

Moment invariants are a classical tool for constructing global image characteristics that are insensitive to a prescribed class of
geometric transformations. Since the work of Hu~\cite{Hu}, an extensive theory of moment invariants has been developed for
groups of rotations, scalings, similarity transformations, and affine
transformations in general. Such invariants have found applications in pattern recognition, image registration, and
object identification.
A systematic account of the theory of moment invariants and their applications is given
in~\cite{FSB}.

Projective transformations present a more difficult case. Unlike the affine action, the Jacobian of a projective transformation depends on
the coordinates of the point, and therefore a finite system of ordinary geometric moments is not closed under this action. This prevents
the direct extension of classical affine constructions to the projective group.
Several approaches have been proposed to overcome this obstacle.
Suk and Flusser~\cite{SukFlusser2004} considered
representations of projective moment invariants
by infinite series involving moments with positive
and negative indices. Another approach uses
additional geometric information. In particular, Wang,
Zhang, and Yao~\cite{WangCoMoments} introduced co-moments
constructed using two reference points.
In this case, finite formulas are obtained at the cost
of having to identify the corresponding points in the images.

An important line of research involves incorporating derivatives
of the image function into moment kernels. Wang and
coauthors~\cite{Wang} proposed a system of
$D$-moments that incorporates coordinates, intensity,
and first derivatives, and used it to construct families
of projective moment expressions. Li, Mo, Xu, and Li~\cite{Li2019} refined this construction,
studied relative projective differential invariants,
and used them as weight functions to construct
finite weighted moment invariants. Their proposed
constructions involve higher-order derivatives of the
intensity function, whose numerical estimation is substantially
more sensitive to noise and discretization errors. Therefore, in this
paper we restrict attention to first-order derivatives, seeking
to preserve projective invariance with less numerical
instability. Higher-order absolute projective image invariants are also considered in the work of Olver~\cite{Olver2023}.

Previous studies have shown that using derivatives
makes it possible to construct finite projective
moment characteristics of images. However, the known constructions
are largely restricted to individual invariant formulas or individual families
of invariants, leaving aside the problem of systematically describing the entire algebra of polynomial
projective differential moment invariants.
This raises a structural problem: to provide a general
algebraic description of all polynomial projective
differential moment invariants and to construct
a systematic mechanism for obtaining them that is not reduced
to searching for individual invariant formulas.
A further question concerns rational moment invariants:
do they form the field of fractions of the algebra of polynomial
invariants, that is, can every rational invariant be
expressed as a ratio of two polynomial invariants?

By its nature, this question belongs to invariant
theory. A similar viewpoint has already proved productive
for ordinary geometric moments. In
the work~\cite{Bedratyuk2020}, algebras of two-dimensional
geometric moment invariants were related
to algebras of joint invariants of binary forms.
This approach replaces the search for individual unwieldy
expressions with the study of a common algebraic structure.

For the projective group, implementing an analogous program
proved substantially more difficult. It was first necessary to develop a theory of
multipoint projective differential
invariants and to identify relative
invariants whose transformation law compensates for
the Jacobian of the projective change of variables.
This problem was addressed successively in three previous papers by the author.
In~\cite{B-2025-1}, the field of multipoint
first-order projective differential invariants
for $n$ points in the plane was described, explicit
generators of the field of absolute invariants were constructed,
and relative invariants associated
with the Jacobian of the projective action were studied.
In the work~\cite{BedratyukFrames}, the same problem
was considered using the moving frame method.
This provided a geometric way of obtaining fundamental
multipoint projective differential
invariants and made it possible to trace directly
the emergence of the relative multiplier associated
with the Jacobian of the projective transformation.
Finally, constructing moment invariants required
a transition from the rational to the polynomial theory. In the work~\cite{B-2026}, classical
$SL(3,\mathbb R)$-invariant theory was used to describe
polynomial multipoint projective
differential invariants and, of particular importance
for the present paper, to characterize relative
invariants of weight $-1$.
Thus, only after this preliminary
invariant-theoretic foundation had been established did it become possible to return to
the original problem of projective moment invariants.

The aim of this paper is to develop a general theory
of polynomial first-order projective differential moment
invariants. To this end, we use
previously obtained results on relative polynomial
invariants of the projective group action on the space of polynomials
in first-jet coordinates. In particular, explicit invariants of low
moment degrees are obtained, their connection
with classical invariant theory of ternary forms is established,
and compact represen\-tations are provided
that are suitable for software computation
of invariant image characteristics.

For a twice-differentiable function $f$ with compact support,
consider the differential moments
$$
m_{\alpha,\beta,r,s}(f)
=
\int_{\mathbb R^2}
x^\alpha y^\beta f_x^r f_y^s f(x,y)\,dx\,dy,
\qquad
\alpha,\beta,r,s\in\mathbb Z_{\geq0}.
$$
The corresponding formal moment symbols
$m_{\alpha,\beta,r,s}$ generate the polynomial algebra
$$
\mathcal{DP}
=
\mathbb R[m_{\alpha,\beta,r,s}].
$$

A projective transformation does not, in general,
preserve this algebra: after a change of variables,
a polynomial moment kernel may become rational,
so the projective group action is not defined on the whole of
$\mathcal{DP}$. At the same time, the action of its Lie algebra
$\mathfrak{pgl}_3\simeq\mathfrak{sl}_3$ is well defined
on the formal moment variables and is given
by derivations of the algebra $\mathcal{DP}$. This allows us
to define the algebra of polynomial projective
differential moment invariants as
$$
\mathcal{MI}
=
\mathcal{DP}^{\mathfrak{sl}_3},
$$
that is, as the subalgebra of moment polynomials annihilated
by all operators of this action.

Classical invariant theory offers two natural approaches
to studying such an algebra. The first consists in identifying
finite-dimensional $\mathfrak{sl}_3$-modules in
$\mathcal{DP}$ and applying standard
invariant-theoretic methods to them. This is precisely how
the module of ten cubic differential moments will be
identified below with the space
of ternary cubics; the classical Aronhold invariants $S$ and $T$
then directly yield two projective moment
invariants. However, this approach requires first
finding the appropriate finite-dimensional module
and therefore does not by itself lead to a general description
of all moment invariants.

The second approach, which is central to this paper, also
has its origins in classical invariant theory. Instead of
searching directly for invariants in the moment algebra,
we pass to a multipoint polynomial
representation of the projective group in which the projective transformation law
has a simpler algebraic structure. This makes it possible
to use the already known description of relative projective
invariants and is conceptually similar to the symbolic method
of classical invariant theory: compact multipoint kernels
are used in place of unwieldy polynomials in
moment variables.

Let $\mathcal P_{n,-1}$ denote the space
of polynomial relative invariants of weight $-1$
for the diagonal projective action on $n$ points,
and let $\mathcal{MI}_n$ denote the space of moment
invariants homogeneous of moment degree $n$.
The main result of this paper states that, for every
$n\geq3$, there is an isomorphism
$$
\Phi_n:
(\mathcal P_{n,-1})^{S_n}
\xrightarrow{} 
\mathcal{MI}_n,
\qquad
\Phi_n^{-1}=\operatorname{Pol}_n.
$$
Here, $(\mathcal P_{n,-1})^{S_n}$ denotes the subspace
of invariants symmetric under permutations of the $n$ points.
The map $\Phi_n$ sends such symmetric invariants to
moment polynomials, while complete polarization
$\operatorname{Pol}_n$ provides the inverse passage to
relative invariants.

Thus, constructing polynomial projective
differential moment invariants reduces to
describing the symmetric part of $\mathcal P_{n,-1}$.
Since the space $\mathcal P_{n,-1}$ itself has already been described, this correspondence
provides a constructive method for obtaining moment invariants.
At the same time, it allows them to be represented by compact
formulas for relative invariants, without expanding the corresponding expressions
into polynomials with hundreds or thousands of moment monomials.

The passage from the formal moment
algebra to differential moments as integral
functionals of $f$ is justified separately. It is proved that,
on the subalgebra containing the constructed invariants,
this integral realization is injective, and hence the passage
from formal moments to their integral definitions does not
produce additional algebraic relations.

Explicit results are obtained for low moment degrees:
nonzero invariants of degrees $1$ and $2$ do not exist,
and a complete description of the invariants of degree $3$
is obtained in terms of the corresponding symmetric relative
projective invariants. The ten cubic
differential moments form a module of ternary cubics,
so the classical Aronhold invariants $S$ and $T$
yield moment invariants of degrees $4$ and $6$.
In particular, $S$ is proportional to the known invariant
$\operatorname{Inv}_{2,0}$ from~\cite{Wang}.

We use the following terminology throughout.
The principal object is called
a \textit{projective differential moment invariant};
within this paper, we use the abbreviated term
\textit{moment invariant}.
For relative invariants depending on several
first jets, we use the term
\textit{multipoint relative projective
differential invariant}.
When the number of points needs to be specified, we call
it \textit{$n$-point};
when the context is unambiguous, we abbreviate the term
to \textit{relative invariant}.

A polynomial $b(x,y,p,q)$ appearing in the integral
$\int_\Omega b(x,y,p,q)f(x,y)\,dx\,dy$
is called a \textit{single-point moment kernel}.
A polynomial $Q(z_1,\ldots,z_n)$ that defines
a moment polynomial through the map $\Phi_n$
is called a \textit{multipoint moment kernel}.
We use the abbreviated term \textit{kernel}
when the number of points is clear from the context.

In this paper, we identify the real Lie groups $PGL(3,\mathbb R)$ and $SL(3,\mathbb R)$, as well as
their Lie algebras $\mathfrak{pgl}_3$ and $\mathfrak{sl}_3$.
For the geometric action on the plane, we use
the notation $PGL(3,\mathbb R)$, while for the corresponding
linear representations and classical invariants, we use
$SL(3,\mathbb R)$.
We denote the operator action on the formal moment algebra
as an action of $\mathfrak{sl}_3$.

The paper is organized as follows. In
Section~\ref{sec:main_results}, the main results
are presented in a form suitable for direct
use. Section~\ref{sec:preliminaries}
provides the necessary background on the projective action
and multipoint relative invariants.
Section~\ref{sec:projective-moment-invariants}
describes the Lie algebra action on the moment algebra,
proves the correspondence theorem, and justifies
the passage from the formal moment algebra to its
integral realization.
Section~\ref{low_order_invariants}
is devoted to explicit constructions of invariants of low
moment degrees, structural properties of the
moment algebra, and a description of rational invariants.

\section{Main Results}
\label{sec:main_results}

In this section, we present the main invariant constructions in a form suitable for direct use.
Proofs of these results and generalizations will be given
in the subsequent sections.

\subsection{Ternary Cubic}

Let $f=f(x,y)$ be a twice-differentiable intensity function
of a grayscale image. Let
$
\Omega=\operatorname{supp}f
$
denote the image domain, and, to simplify the subsequent formulas, set
$$
p=f_x,\qquad q=f_y,\qquad w=xp+yq.
$$

Consider the ten differential moments
\begin{equation*}
\mu_{ijk}
=
\int_\Omega p^iq^j(-w)^k f(x,y)\,dx\,dy,
\qquad i+j+k=3,
\end{equation*}
which we will regard as the coefficients of a ternary cubic
\begin{equation*}
\begin{gathered}
\mathcal F(\xi,\eta,\zeta)
=
\mu_{300}\xi^3+\mu_{030}\eta^3+\mu_{003}\zeta^3
+3\mu_{210}\xi^2\eta+3\mu_{120}\xi\eta^2
+3\mu_{201}\xi^2\zeta
\\
+3\mu_{102}\xi\zeta^2
+3\mu_{021}\eta^2\zeta
+3\mu_{012}\eta\zeta^2
+6\mu_{111}\xi\eta\zeta.
\end{gathered}
\end{equation*}
Under this identification, the space spanned by the coefficients
$\mu_{ijk}$ is a $10$-dimensional $SL(3,\mathbb R)$-module
dual to the classical module of ternary cubics in the variables
$\xi,\eta,\zeta$.
The problem of constructing projective invariants as polynomials in these
ten moments reduces to the problem of constructing
invariants of a ternary cubic. Thus, it can be solved
by directly applying known results from classical invariant theory;
see~\cite{Aronhold1850}:
\begin{theorem}
All polynomial invariants of a ternary cubic, viewed as polynomials in $\mu_{ijk}$, are generated by two algebraically independent
invariants of degree $4$ and degree $6.$
\end{theorem}
Denote these two generating invariants by $S$ and $T$, and, to express them explicitly, set:

\begin{gather*}
A=\mu_{300}, \qquad B=\mu_{030}, \qquad C=\mu_{003}, \qquad
D=\mu_{210}, \qquad E=\mu_{021}, \\
F=\mu_{102}, \qquad
G=\mu_{120}, \qquad H=\mu_{012}, \qquad I=\mu_{201}, \qquad J=\mu_{111}.
\end{gather*}

We use the notation $A,B,\ldots,J$
as local abbreviations for these moments
only in the explicit polynomial formulas for the invariants
$S$ and $T$, including the formula in the appendix.
Elsewhere in the text, we retain
the indexed notation $\mu_{ijk}$.

The first invariant then takes the form
\begin{align}
S={}&
AGEC-AGH^2-AJBC+AJEH+AFBH-AFE^2
\nonumber\\
&-D^2EC+D^2H^2+DIBC-DIEH+DGJC-DGFH
\nonumber\\
&-2DJ^2H+3DJFE-DF^2B-I^2BH+I^2E^2
\nonumber\\
&-IG^2C+3IGJH-IGFE-2IJ^2E+IJFB
\nonumber\\
&+G^2F^2-2GJ^2F+J^4.
\label{eq:Aronhold_S}
\end{align}
It contains $25$ monomials and is proportional to the invariant
$\operatorname{Inv}_{2,0}$ from~\cite{Wang}.
The second invariant, $T$, contains $103$ monomials; its formula
is given in the appendix; see~\eqref{eq:Aronhold_T}.

Thus, the known 25-term expression is one of the two
fundamental invariants of this moment module.
All other projective moment invariants obtained have substantially more unwieldy explicit expansions in differential moments. Some of the invariants constructed in this paper are implemented
in the accompanying software repository, which allows them
to be computed without explicitly expanding unwieldy
polynomial formulas. Therefore, instead of full formulas,
we use their compact symbolic representations, which can, when necessary, be expanded into explicit polynomials
in differential moments.

\subsection{Projective Moment Invariants of Degree Three}
\label{subsec:degree3_results}

For the full projective group, nonzero moment invariants of degrees $1$ and $2$
do not exist, and the lowest degree at which new
invariants arise is degree $3$.
To write them compactly, we use three points
$$
z_i=(x_i,y_i,p_i,q_i), \quad p_i=\partial_{x_i}f(x_i,y_i), \quad q_i=\partial_{y_i}f(x_i,y_i), \quad w_i=x_ip_i+y_iq_i,
\qquad i=1,2,3,
$$

Set
$$
\ell_i=(p_i,q_i,-w_i),
\qquad
v_i=(x_i,y_i,1).
$$
For distinct indices $i,j,k$, define
$$
D_{ijk}
=
\det
\begin{pmatrix}
p_i&q_i&-w_i\\
p_j&q_j&-w_j\\
p_k&q_k&-w_k
\end{pmatrix},
\qquad
C_{ij}=\ell_i\cdot v_j
=p_ix_j+q_iy_j-w_i.
$$
In formal expansions in $D$-moments, the quantities
$w_i$ are temporarily treated as independent symbols.
It is for this form that we subsequently report the number
of monomials, the number of $D$-moments used,
and the $D$-order of the representation.
When realizing these expressions on first jets, we substitute
$w_i=x_ip_i+y_iq_i$, so that
$$
C_{ij}
=
p_i(x_j-x_i)+q_i(y_j-y_i).
$$
Thus, the independence of the symbols $w_i$ is merely
a convention for the formal expansion;
the geometric transformation laws apply
after the indicated substitution.

To pass from multipoint moment kernels
to moment polynomials, we use
the map $\Phi_n$.
In terms of $D$-moments (see Wang~\cite{Wang}),
$$
\phi_{\alpha\beta rst}
=
\int_\Omega x^\alpha y^\beta p^rq^sw^t f(x,y)\,dx\,dy
$$
it is defined by
$$
\Phi_n\left(
\prod_{\nu=1}^{n}
x_\nu^{\alpha_\nu}y_\nu^{\beta_\nu}
p_\nu^{r_\nu}q_\nu^{s_\nu}w_\nu^{t_\nu}
\right)
=
\prod_{\nu=1}^{n}
\phi_{\alpha_\nu\beta_\nu r_\nu s_\nu t_\nu}
$$
and on arbitrary multipoint polynomial kernels by linear extension;
for example,
$$
\Phi_3(x_1p_1^2\,y_2q_2^2\,p_3q_3w_3)
=
\phi_{10200}\phi_{01020}\phi_{00111}.
$$

This definition of $\Phi_n$ agrees with its
definition on polynomials in the variables $z_i$,
which is used below.
Indeed, substituting $w=xp+yq$ gives
$$
\phi_{\alpha\beta rst}
=
\sum_{h=0}^{t}\binom{t}{h}
m_{\alpha+h,\,\beta+t-h,\,r+h,\,s+t-h}.
$$
In the formal moment algebra, we take this equality
as the definition of the symbol $\phi_{\alpha\beta rst}$.
Thus, one may first apply $\Phi_n$
in the $D$-moment notation and then expand
$\phi$ in terms of $m$, or first substitute
$w_i=x_ip_i+y_iq_i$ into the kernel and apply $\Phi_n$
to the resulting polynomial.
Both procedures yield the same element
of $\mathcal{DP}_n$.

We subsequently use multipoint kernels
to represent moment constructions. A short formula for such a kernel
in terms of the determinants $D_{ijk}$ and the cycle factors $C_{ij}$
can define a moment expansion with hundreds or thousands
of terms. Therefore, general constructions are presented primarily
through these symbolic formulas, while invariants
of low moment degrees are considered in greater detail
as explicit examples of the general approach.
In this notation, both invariants of the ternary cubic
have compact representations

\begin{align*}
S&=
\frac1{24}\,
\Phi_4(D_{123}D_{124}D_{134}D_{234}),\\
T&=
-\frac16\,
\Phi_6(D_{123}D_{124}D_{135}D_{236}D_{456}^{\,2}).
\end{align*}

For three points, introduce cycle constructions, that is, products
of the quantities $C_{ij}$ in which the indices form a closed cycle.
In particular,
\begin{equation}
\label{dva}
\mathsf X=C_{12}C_{21},\qquad
\mathsf Y=C_{13}C_{31},\qquad
\mathsf Z=C_{23}C_{32}.
\end{equation}
correspond to 2-cycles, while
\begin{equation}
\label{try}
U=C_{12}C_{23}C_{31},
\qquad
V=C_{13}C_{32}C_{21}
\end{equation}
correspond to the two possible oriented 3-cycles.

These cycle quantities satisfy the relation
$$
UV
=
(C_{12}C_{23}C_{31})
(C_{13}C_{32}C_{21})
=
\mathsf X \mathsf Y \mathsf Z.
$$
Thus, $\mathsf X,\mathsf Y,\mathsf Z,U,V$ are not algebraically independent, and
the representation of a polynomial in the form
$Q(\mathsf X,\mathsf Y,\mathsf Z,U,V)$ is generally not unique. Hereafter, such notation
will be understood subject to the relation
$
UV=\mathsf X \mathsf Y \mathsf Z.
$
Nevertheless, these five cycle quantities form a convenient
generating system for describing the relative invariants
that arise in the three-point case.

\begin{theorem}
\label{all_degree3}
Every invariant $I\in\mathcal{MI}_3$ has the form
$$
I=
\Phi_3\left(D_{123}^{\,3}Q(\mathsf X,\mathsf Y,\mathsf Z,U,V)\right),
$$
where $Q$ is antisymmetric under permutations of the labels
of the three points.
Conversely, every such expression is a projective
differential moment invariant.
\end{theorem}
Antisymmetry means that the polynomial $Q$ changes sign under every odd permutation
of the three points. A complete justification
of this description and a proof that it includes all invariants
of moment degree three are given in
Section~\ref{subsec:degree_three}.

\medskip

A particularly simple family is obtained from polynomials
in $U,V$ alone. Consider the vector subspace
$$
W_a
=
\operatorname{span}
\left\{
\phi_{\alpha\beta rst}
\mid
\alpha+\beta\leq a,\quad r+s+t=a+3
\right\}
\subset\mathcal{DP}_1,
\qquad a\geq0.
$$
This generating system is generally
not linearly independent. Indeed, the equality
$w=xp+yq$ implies
$$
\phi_{\alpha\beta rs,t+1}
=
\phi_{\alpha+1,\beta,r+1,s,t}
+
\phi_{\alpha,\beta+1,r,s+1,t}.
$$
Therefore, hereafter $\operatorname{Sym}^3(W_a)$
will denote the subspace of $\mathcal{DP}_3$ spanned by
all products
$$
M_1M_2M_3,
\qquad
M_1,M_2,M_3\in W_a.
$$
The following theorem presents an infinite two-parameter family of projective differential moment invariants of degree $3$, parameterized by the integers $a$ and $k$.

\begin{theorem}
\label{J-family}
For every $a\geq1$, the invariants
$$
J_{a,k}
=
\Phi_3\left[
D_{123}^{\,3}
\left(U^{a-k}V^k-U^kV^{a-k}\right)
\right],
\qquad
k=0,\ldots,\left\lfloor\frac{a-1}{2}\right\rfloor,
$$
form a basis of the space of polynomial moment invariants in
$\operatorname{Sym}^3(W_a)$.
\end{theorem}

The simplest case, $a=1$, $k=0$, gives
\begin{equation}
\label{eq:J10}
J_{1,0}
=
I_5^{(3)}
=
\Phi_3\left[
D_{123}^{\,3}
\left(
C_{12}C_{23}C_{31}-C_{13}C_{32}C_{21}
\right)
\right].
\end{equation}
Its displayed $D$-moment representation has order
$5$ and contains $342$ monomials. Other invariants of this family and
constructions from mixed homogeneous classes
are considered in Section~\ref{low_order_invariants}.

\subsection{General Moment Correspondence}
\label{subsec:higher_degree_results}

This construction extends to an arbitrary moment
degree $n\geq1$.

A polynomial $Q(z_1,\ldots,z_n)$ is called
a \textit{multipoint relative projective
differential invariant of weight $-1$} if
the identity
\begin{equation}
\label{eq:weight_minus_one_simple}
Q(gz_1,\ldots,gz_n)
=
\left(
\prod_{i=1}^{n}j_g(x_i,y_i)
\right)^{-1}
Q(z_1,\ldots,z_n),
\end{equation}
holds for every projective transformation $g$,
where $j_g(x_i,y_i)$ is the signed Jacobian of the projective
transformation at the point $(x_i,y_i)$.

The following main assertion holds.

\begin{theorem}
\label{thm:general_main_result}
Let $n\geq1$, and let $Q(z_1,\ldots,z_n)$ be a symmetric relative
projective invariant of weight $-1$. Then
$
\Phi_n(Q)
$
is a projective differential moment invariant
of moment degree $n$.
Conversely, every projective differential moment
invariant $I$ of moment degree $n$ can be written as
$
I=\Phi_n(Q),
$
where $Q$ is a symmetric relative projective invariant
of weight $-1$ in $n$ sets of variables.
\end{theorem}
This theorem establishes a connection between projective
moment invariants and symmetric multipoint
relative invariants of weight $-1$. A description of such
relative invariants in terms of admissible monomials in
$D_{ijk}$ and $C_{ij}$ is given in
Section~\ref{sec:preliminaries}, while the proof of the
theorem itself is given in Section~\ref{sec:projective-moment-invariants}.

For an admissible pair $(g,f)$, let
$\varepsilon_g=\operatorname{sgn}j_g(x,y)$
denote the constant sign of the Jacobian on a connected domain containing
the support of $f$. Then, for every
$I\in\mathcal{MI}_n$, we have
$$
I(g\cdot f)=\varepsilon_g^{\,n}I(f).
$$
Indeed, let $I=\Phi_n(Q)$, where $Q$ is a symmetric
relative invariant of weight $-1$.
A change of variables in the multiple integral and the
relative invariance law give
$$
\begin{aligned}
I(g\cdot f)
&=
\int_{\Omega^n}
Q(gz_1,\ldots,gz_n)
\prod_{i=1}^n
\bigl(|j_g(x_i,y_i)|f(x_i,y_i)\,dx_i\,dy_i\bigr)
\\
&=
\int_{\Omega^n}
Q(z_1,\ldots,z_n)
\prod_{i=1}^n
\left(
\frac{|j_g(x_i,y_i)|}{j_g(x_i,y_i)}
f(x_i,y_i)\,dx_i\,dy_i
\right)
=
\varepsilon_g^{\,n}I(f).
\end{aligned}
$$
Thus, invariants of even moment degree
are preserved under all admissible transformations,
while invariants of odd degree change sign under orientation reversal
and are pseudoinvariants.

In addition to the correspondence theorem, two
structural properties of moment invariants are established.
First, the space $\mathcal{MI}_3$ is
infinite-dimensional, and the algebra $\mathcal{MI}$
is not finitely generated. Thus,
finite generation of each module
$\mathcal P_{n,-1}$ over the algebra of absolute
invariants $\mathcal A_n$ for fixed
$n\geq3$ does not imply finite generation
of the algebra of all moment invariants.

Second, every rational moment
invariant can be written in the form
$$
\frac{P}{Q},
\qquad
P,Q\in\mathcal{MI},\quad Q\neq0.
$$
Thus, the field of rational moment invariants
is the field of fractions of the algebra $\mathcal{MI}$.
Proofs of both results are given
in Section~\ref{low_order_invariants}.

\section{Preliminaries}
\label{sec:preliminaries}


In this section, we introduce the projective group action on the image
coordinates and its first derivatives, the basic
constructions $D_{ijk}$ and $C_{ij}$, and their transformation
laws, which will subsequently be used to construct
projective differential moment invariants.

\subsection{The projective group and its action}

A projective transformation of the plane is determined, up to
a nonzero scalar factor, by a real nonsingular matrix
$$
H=
\begin{pmatrix}
a_1&a_2&a_3\\
b_1&b_2&b_3\\
c_1&c_2&c_3
\end{pmatrix}
\in GL(3,\mathbb R).
$$
Let $g=[H]\in PGL(3,\mathbb R)$ denote the corresponding equivalence class
of the matrix. In an affine chart, the coordinate map $h_g$ takes the form
$$
h_g(x,y)=(\widetilde x,\widetilde y)=
\left(
\frac{a_1x+a_2y+a_3}{s},
\frac{b_1x+b_2y+b_3}{s}
\right),
\qquad
s=c_1x+c_2y+c_3,
$$
and is defined whenever $s\neq0$. Its signed Jacobian is
$$j_g(x,y)=\det Dh_g(x,y)=\frac{\det H}{s^3}. $$ All such transformations form a group called
the \textit{projective group of the plane},
$PGL(3,\mathbb R)$. Following the convention adopted in the introduction,
we identify this group with $SL(3,\mathbb R)$
and its Lie algebra with $\mathfrak{sl}_3$.

In applications, the normalization $c_3=1$ is often used.
Since it is not available for every transformation, and the corresponding
matrix representatives are not closed under multiplication, we do not
impose it hereafter.

\medskip

We describe an image by a real-valued intensity function
$f=f(x,y)$. We assume that it
is twice differentiable on its domain
and vanishes outside it.
The support of the function, that is, the closure of the set of points where
$f(x,y)\neq0$, will be denoted by $\Omega$.

A projective transformation cannot be applied to every image: its denominator $s=c_1x+c_2y+c_3$
may vanish on the support of the function $f$.
A pair $(g,f)$ consisting of a projective transformation $g$ and an image $f$ to which it is applied
will be called \textit{admissible} if there exists an open connected domain $U$ containing the entire support $\Omega$ of the image and on which $s\neq0$.
Under this condition, the transformation is defined throughout $U$, and its Jacobian has a constant nonzero sign there.
Thus, we distinguish the signed Jacobian in the transformation laws
of algebraic kernels from its absolute value in the
change-of-variables formula.

For an admissible pair $(g,f)$, we define the action of the projective
transformation on the image by
$$
(g\cdot f)\bigl(h_g(x,y)\bigr)=f(x,y),
\qquad (x,y)\in U.
$$
Thus, on $h_g(U)$, we have
$$
g\cdot f=f\circ h_g^{-1},
$$
and outside this domain, we extend the function $g\cdot f$
by zero.

The transformation law of a moment invariant
under the projective group action depends on
its moment degree and the sign of the Jacobian,
as specified by the following assertion.

\begin{theorem}
Let $(g,f)$ be an admissible pair, and let
$$
\varepsilon_g=\operatorname{sgn}j_g(x,y),
\qquad (x,y)\in U,
$$
be the constant sign of the Jacobian on the domain $U$.
Then, for every $I\in\mathcal{MI}_n$, we have
$$
I(g\cdot f)=\varepsilon_g^{\,n}I(f).
$$
\end{theorem}

\begin{proof}
Let $I=\Phi_n(Q)$, where $Q$ is a symmetric
relative invariant of weight $-1$.
A change of variables in the multiple integral and the
relative invariance law give
$$
\begin{aligned}
I(g\cdot f)
&=
\int_{\Omega^n}
Q(gz_1,\ldots,gz_n)
\prod_{i=1}^n
\bigl(
|j_g(x_i,y_i)|f(x_i,y_i)\,dx_i\,dy_i
\bigr)=
\\
&=
\int_{\Omega^n}
Q(z_1,\ldots,z_n)
\prod_{i=1}^n
\left(
\frac{|j_g(x_i,y_i)|}{j_g(x_i,y_i)}
f(x_i,y_i)\,dx_i\,dy_i
\right)
=\varepsilon_g^{\,n}I(f).
\end{aligned}
$$
\end{proof}

Thus, invariants of even moment degree
are preserved under all admissible transformations,
while invariants of odd moment degree
change sign under orientation reversal
and are pseudoinvariants.

The projective group action is naturally prolonged to first derivatives:
\begin{equation}\label{pq}
\begin{gathered}
p\longmapsto
\widetilde p
=
\frac{s}{\det H}
\Bigl[
-(b_1c_2-b_2c_1)(px+qy)
+(b_2c_3-b_3c_2)p
-(b_1c_3-b_3c_1)q
\Bigr],
\\
q\longmapsto
\widetilde q
=
\frac{s}{\det H}
\Bigl[
(a_1c_2-a_2c_1)(px+qy)
-(a_2c_3-a_3c_2)p
+(a_1c_3-a_3c_1)q
\Bigr],
\end{gathered}
\end{equation}
see, for example,~\cite{B-2025-1}.
\subsection{Diagonal action}

Consider $n$ independent sets of variables
$z_i=(x_i,y_i,p_i,q_i)$, $i=1,\ldots,n$, and denote the corresponding partial derivatives by $p_i=f_x(x_i,y_i)$ and $q_i=f_y(x_i,y_i)$.

The diagonal action means that the same element $g=[H]$
acts on each set $z_i$. In particular,
$$
(x_i,y_i)\longmapsto
(\widetilde x_i,\widetilde y_i)=
\left(
\frac{a_1x_i+a_2y_i+a_3}{s_i},
\frac{b_1x_i+b_2y_i+b_3}{s_i}
\right),
\qquad
s_i=c_1x_i+c_2y_i+c_3,
$$
and the formulas for $p_i,q_i$ are obtained from~\eqref{pq} by replacing
$(x,y,p,q,s)$ with $(x_i,y_i,p_i,q_i,s_i)$.

We denote the coordinate Jacobian multiplier of the diagonal action
by
$$
\mathcal J_g(\boldsymbol z)
=
\prod_{i=1}^{n}j_g(x_i,y_i)
=
\frac{(\det H)^n}{\prod_{i=1}^{n}s_i^3},
\qquad
\boldsymbol z=(z_1,\ldots,z_n).
$$
It accounts for the change of coordinates $(x_i,y_i)$
and is used in the relative invariance laws.
The change-of-variables formula for integrals
uses its absolute value
$$
|\mathcal J_g(\boldsymbol z)|
=
\prod_{i=1}^{n}|j_g(x_i,y_i)|.
$$
For the prolonged action on all $4n$ variables
$z_i$, the determinant of the Jacobian matrix
equals $1$.

\subsection{Relative and absolute invariants}

Let $R_n=\mathbb R[\boldsymbol z]$ denote the polynomial
ring, and let $K_n=\operatorname{Frac}(R_n)$ denote its
field of fractions. Projective substitutions are rational:
they are defined in $K_n$ but generally do not preserve $R_n$.
A rational function $F\in K_n$ is called
a \textit{multipoint relative projective
differential invariant of weight $\omega$} if the identity
$$
F(g\boldsymbol z)
=
\mathcal J_g(\boldsymbol z)^\omega F(\boldsymbol z).
$$
holds for all $g\in PGL(3,\mathbb R)$.
Hereafter, it suffices to consider integer weights
$\omega\in\mathbb Z$, in particular $\omega=0$ and $\omega=-1$.
Invariants of weight $0$ are called \textit{absolute}.

We denote by $\mathcal P_{n,\omega}\subset R_n$
the vector space of polynomial relative invariants of weight $\omega$,
and by $\mathcal A_n=\mathcal P_{n,0}$ the algebra
of absolute polynomial invariants. For each $\omega$,
the space $\mathcal P_{n,\omega}$ is a module over $\mathcal A_n$.

Polynomial absolute invariants and modules of relative
invariants were studied in~\cite{B-2026}. We recall the results
needed to describe $\mathcal P_{n,-1}$.
The determinants $D_{ijk}$ and the quantities $C_{ij}$,
introduced in Subsection~\ref{subsec:degree3_results},
satisfy the transformation laws
$$
D_{ijk}\longmapsto
\frac{s_is_js_k}{\det H}D_{ijk},
\qquad
C_{ij}\longmapsto\frac{s_i}{s_j}C_{ij}.
$$

For a cycle $\gamma=(i_1\,i_2\,\ldots\,i_r)$,
$r\geq2$, in the symmetric group $S_n$, define the cyclic invariant
$$
Z_\gamma
=
C_{i_1i_2}C_{i_2i_3}\cdots
C_{i_{r-1}i_r}C_{i_ri_1}.
$$

\begin{theorem}[\cite{B-2026}, Theorem 2.2]
For every $n\geq3$, the algebra $\mathcal A_n$ is generated
by cyclic invariants:
$$
\mathcal A_n
=
\mathbb R\bigl[
Z_\gamma
\mid
\gamma\in S_n
\text{ --- a cycle of length }r\geq2
\bigr].
$$
\end{theorem}

To describe $\mathcal P_{n,-1}$, consider the set
$\mathcal T_n=\{I\subset[n]\mid |I|=3\}$, where
$[n]=\{1,\ldots,n\}$, and, for $I=\{i,j,k\}$ with $i<j<k$,
set $D_I=D_{ijk}$.

Consider a vector of length $\binom{n}{3}$
$$
\boldsymbol b=(b_I)_{I\in\mathcal T_n},
\qquad
b_I\in\mathbb Z_{\geq0},
$$
and a nonnegative integer $n\times n$ matrix with zero diagonal
$$
\boldsymbol a=(a_{ij}),
\qquad
a_{ij}\in\mathbb Z_{\geq0}.
$$

Define the monomial
\begin{equation}
\label{Pba}
P_{\boldsymbol b,\boldsymbol a}
=
\prod_{I\in\mathcal T_n}D_I^{b_I}
\prod_{i\neq j}C_{ij}^{a_{ij}}.
\end{equation}
Its $C$-degree is defined as the number $\sum_{i\neq j}a_{ij}$.
For each $i$, set
$$
d_{\boldsymbol b}(i)=\sum_{I\ni i}b_I,
\qquad
\partial_{\boldsymbol a}(i)
=
\sum_{j\neq i}a_{ij}-\sum_{j\neq i}a_{ji}.
$$
The first quantity accounts for the total multiplicity of occurrences of the index
$i$ in the determinants, while the second gives the balance of its occurrences
in the factors $C_{ij}$.
The monomial~\eqref{Pba} is called \textit{admissible} if
\begin{equation}
\label{admissibility_conditions}
\sum_{I\in\mathcal T_n}b_I=n,
\qquad
d_{\boldsymbol b}(i)+\partial_{\boldsymbol a}(i)=3,
\qquad i=1,\ldots,n.
\end{equation}

The following structural result holds.
\begin{theorem}[\cite{B-2026}, Theorem 3.2]
For every $n\geq3$, the vector space
$\mathcal P_{n,-1}$ is spanned by the admissible monomials
$P_{\boldsymbol b,\boldsymbol a}$.
\end{theorem}

Since multiplying a relative invariant of weight $-1$
by an absolute invariant does not change its weight,
the space $\mathcal P_{n,-1}$ is naturally a module over the algebra
$\mathcal A_n$. The following result holds for this module.

\begin{theorem}[\cite{B-2026}, Theorem 4.2]
For every $n\geq3$, the module $\mathcal P_{n,-1}$ is finitely
generated over $\mathcal A_n$. A finite generating system
can be chosen among the admissible monomials.
\end{theorem}

\section{Projective Differential Moment Invariants}
\label{sec:projective-moment-invariants}

In this section, we pass from the multipoint projective
constructions introduced above to differential moments
of an image. We first define the corresponding moment variables and
the map that sends multipoint polynomial kernels
to polynomials in moments. We then state and prove
the main result establishing a correspondence between
relative projective invariants of weight $-1$ and
projective differential moment invariants.

\subsection{Differential moments}
The differential moments of an image $f$
are defined as the quantities
\begin{equation}
\label{differential-moments}
m_{\alpha,\beta,r,s}(f)
=
\int_\Omega x^\alpha y^\beta p^r q^s f(x,y)\,dx\,dy,
\qquad
\alpha,\beta,r,s\in\mathbb Z_{\geq0}.
\end{equation}
Hereafter, we distinguish the following characteristics:
\begin{itemize}
\item \textit{moment degree} --- the total polynomial
degree in the moment variables; a moment monomial
of degree $n$ contains $n$ moment factors,
counted with their multiplicities;
\item \textit{moment order}
of $m_{\alpha,\beta,r,s}$ --- the number $\alpha+\beta+r+s$;
\item \textit{$D$-order} of a moment
$\phi_{\alpha\beta rst}$ --- the number $\alpha+\beta+r+s+t$,
while the $D$-order of a particular representation is
the maximum of these numbers over its moment factors;
\item \textit{differential order} --- the highest
order of derivatives of the function $f$ used in the kernel;
only first derivatives are used in this paper;
\item \textit{derivative degree} of an individual moment
kernel --- its polynomial degree in $p,q$:
$r+s$ for $m_{\alpha,\beta,r,s}$ and $r+s+t$
for $\phi_{\alpha\beta rst}$ after substituting
$w=xp+yq$.
\end{itemize}
For example, the invariant $S$ has moment degree $4$
and the displayed $D$-moment representation has order $3$.
Each of its moment factors has derivative degree
$3$, although only derivatives of first
differential order are used.

We denote the corresponding single-point moment kernel
by $b_{\alpha,\beta,r,s}(z)=x^\alpha y^\beta p^r q^s$,
where $z=(x,y,p,q)$.

For an admissible pair $(g,f)$, a change of variables gives
$$
m_{\alpha,\beta,r,s}(g\cdot f)
=
\int_\Omega
b_{\alpha,\beta,r,s}(gz)\,
|j_g(x,y)|\,f(x,y)\,dx\,dy.
$$

Since the kernel generally becomes rational
after a projective transformation, the transformed
moment need not be expressible polynomially
in terms of the original moments. In contrast, the Lie algebra action
of the projective group has a substantially simpler structure:
it sends each moment to a linear
combination of moments.

Consider the formal algebra
$$
\mathcal{DP}
=
\mathbb R[
m_{\alpha,\beta,r,s}
\mid
\alpha,\beta,r,s\in\mathbb Z_{\geq0}
],
$$
where the moment symbols $m_{\alpha,\beta,r,s}$ are algebraically independent variables.
We denote by $\mathcal{DP}_n$ the homogeneous component of
degree $n$ in $\mathcal{DP}$.
The algebra $\mathcal{DP}$ is the principal algebraic object
in which we will study projective differential
moment invariants. The formal approach allows us to study
the Lie algebra action on the moment variables
instead of considering the projective group action directly.

For a specific image $f$, we replace each formal moment
symbol $m_{\alpha,\beta,r,s}$ with the numerical value
of the corresponding differential moment defined by
formula~\eqref{differential-moments}. This substitution
defines a homomorphism
$
\operatorname{ev}_f:\mathcal{DP}\to\mathbb R.
$
Thus, if $I\in\mathcal{DP}$ is a polynomial in the moment
symbols, substituting the moments of the image $f$
gives the number
$$
I(f)=\operatorname{ev}_f(I).
$$

\subsection{Lie algebra action}

In this subsection, we derive explicit formulas for the action of the Lie algebra
$\mathfrak{sl}_3$
on differential moments and on
the moment algebra $\mathcal{DP}$.
Let
$$
X=\xi(x,y)\partial_x+\eta(x,y)\partial_y
$$
be an element of the Lie algebra $\mathfrak{sl}_3$ realized
as a vector field on the plane. Its first prolongation
to the first-jet space with coordinates $(x,y,p,q)$
takes the form
\begin{equation}
\label{first_prolong}
X^{(1)}
=
\xi\partial_x+\eta\partial_y
-(\xi_xp+\eta_xq)\partial_p
-(\xi_yp+\eta_yq)\partial_q,
\end{equation}
see, for example,~\cite{Olver1993LieGroups}.
When passing to moment integrals, the
change-of-variables formula introduces the multiplier $|j_g(x,y)|$.
For the one-parameter subgroup $g_t$
corresponding to the vector field $X$, we have
$j_{g_0}(x,y)=1$, and hence
$$
\left.\frac{d}{dt}|j_{g_t}(x,y)|\right|_{t=0}
=
\left.\frac{d}{dt}j_{g_t}(x,y)\right|_{t=0}
=
\operatorname{div}X
=
\xi_x+\eta_y.
$$
Thus, on single-point kernels, we obtain the operator
$$
\mathscr L_X b
=
X^{(1)}b+(\xi_x+\eta_y)b.
$$

Using the special case of the map $\Phi_n$
introduced in Section~\ref{sec:main_results}, for a monomial $b=b_{\alpha,\beta,r,s}$, set
$m_b=\Phi_1(b)$.
We define the action on generators by the rule
$$\mathscr D_Xm_b=\Phi_1(\mathscr L_Xb),$$ and extend it to
$\mathcal{DP}$ as a derivation.

In its standard realization on the affine plane, the Lie algebra
$\mathfrak{sl}_3$ is an eight-dimensional algebra of vector fields.
We use the following basis, which is equivalent to the classical basis
in~\cite[p. 26]{SL-1}:
\begin{equation*}
\begin{aligned}
X_1&=\partial_x, &
X_2&=\partial_y, &
X_3&=x\partial_x, &
X_4&=y\partial_y,\\
X_5&=y\partial_x, &
X_6&=x\partial_y, &
X_7&=x^2\partial_x+xy\partial_y, &
X_8&=xy\partial_x+y^2\partial_y.
\end{aligned}
\end{equation*}

\begin{example}
Consider the basis element
$X_7=x^2\partial_x+xy\partial_y$.
Here, $\xi=x^2$, $\eta=xy$, so
$\xi_x=2x$, $\xi_y=0$, $\eta_x=y$, and $\eta_y=x$.
The first prolongation formula~\eqref{first_prolong} gives
$$
X_7^{(1)}
=
x^2\partial_x+xy\partial_y
-(2xp+yq)\partial_p-xq\partial_q.
$$
Since $\operatorname{div}X_7=2x+x=3x$,
the weighted operator on moment kernels takes the form
$$
\mathscr L_{X_7}
=
x^2\partial_x+xy\partial_y
-(2xp+yq)\partial_p-xq\partial_q+3x.
$$
Therefore,
\begin{align*}
\mathscr D_7m_{\alpha,\beta,r,s}
&=
\Phi_1\!\left[
\left(
x^2\partial_x+xy\partial_y
-(2xp+yq)\partial_p-xq\partial_q+3x
\right)
x^\alpha y^\beta p^rq^s
\right]
\\
&=
\Phi_1\!\left[
(\alpha+\beta-2r-s+3)
x^{\alpha+1}y^\beta p^rq^s
-rx^\alpha y^{\beta+1}p^{r-1}q^{s+1}
\right]
\\
&=
(\alpha+\beta-2r-s+3)
m_{\alpha+1,\beta,r,s}
-r\,m_{\alpha,\beta+1,r-1,s+1}.
\end{align*}

\end{example}

Similarly, for all $\mathscr D_k=\mathscr D_{X_k}$, we obtain
\begin{equation}
\label{moment-operators}
\begin{aligned}
\mathscr D_1m_{\alpha,\beta,r,s}
&=\alpha m_{\alpha-1,\beta,r,s},\\
\mathscr D_2m_{\alpha,\beta,r,s}
&=\beta m_{\alpha,\beta-1,r,s},\\
\mathscr D_3m_{\alpha,\beta,r,s}
&=(\alpha-r+1)m_{\alpha,\beta,r,s},\\
\mathscr D_4m_{\alpha,\beta,r,s}
&=(\beta-s+1)m_{\alpha,\beta,r,s},\\
\mathscr D_5m_{\alpha,\beta,r,s}
&=\alpha m_{\alpha-1,\beta+1,r,s}
-s\,m_{\alpha,\beta,r+1,s-1},\\
\mathscr D_6m_{\alpha,\beta,r,s}
&=\beta m_{\alpha+1,\beta-1,r,s}
-r\,m_{\alpha,\beta,r-1,s+1},\\
\mathscr D_7m_{\alpha,\beta,r,s}
&=(\alpha+\beta-2r-s+3)m_{\alpha+1,\beta,r,s}
-r\,m_{\alpha,\beta+1,r-1,s+1},\\
\mathscr D_8m_{\alpha,\beta,r,s}
&=(\alpha+\beta-r-2s+3)m_{\alpha,\beta+1,r,s}
-s\,m_{\alpha+1,\beta,r+1,s-1}.
\end{aligned}
\end{equation}
Terms with zero coefficients are omitted, so moments
with negative indices do not arise.

The explicit formulas
show that the action of $\mathfrak{sl}_3$ on the moment algebra
is compatible with two natural graded structures.
In particular, it does not change the number of moment factors
in a polynomial and, as follows from~\eqref{moment-operators}, preserves the total derivative degree $r+s$
of each moment generator.

A polynomial $I\in\mathcal{DP}$ is called
a \textit{projective differential moment invariant}
if it is invariant under the specified action
of the Lie algebra $\mathfrak{sl}_3$, that is,
$$
\mathscr D_k I=0,
\qquad k=1,\ldots,8.
$$

The set of all projective differential moment
invariants forms the subalgebra
$$
\mathcal{MI}
=
\mathcal{DP}^{\mathfrak{sl}_3}
=
\bigcap_{k=1}^{8}\ker\mathscr D_k.
$$
We denote its homogeneous component of moment degree $n$
by
$
\mathcal{MI}_n
=
\mathcal{MI}\cap\mathcal{DP}_n .
$

Formally, the problem of constructing moment invariants now reduces
to finding common solutions of the system of linear
first-order differential equations~\eqref{moment-operators}.
The method of characteristics can provide a local description of such solutions
in terms of orbit invariants of the corresponding Lie algebra action; however,
by itself, it does not provide a general description of those solutions
that are polynomials in differential moments. Therefore, below,
we use an algebraic approach based on
the map $\Phi_n$, polarization, and multipoint
relative projective invariants.

\subsection{Moment realization and polarization}

Every monomial in $R_n$ has the form
$$\prod_{i=1}^{n}b_{\alpha_i,\beta_i,r_i,s_i}(z_i).$$
Define the linear map
$\Phi_n:R_n\to\mathcal{DP}_n$ by the rule
$$
\Phi_n\left(
\prod_{i=1}^{n}b_{\alpha_i,\beta_i,r_i,s_i}(z_i)
\right)
=
\prod_{i=1}^{n}m_{\alpha_i,\beta_i,r_i,s_i}.
$$
After substituting the moments of an image, it is realized
by the multiple integral
\begin{equation}
\label{Phi-n}
\Phi_n(F)(f)
=
\int_{\Omega^n}
F(\boldsymbol z)
\prod_{i=1}^{n}\bigl(f(x_i,y_i)\,dx_i\,dy_i\bigr).
\end{equation}

By the definition of $\Phi_n$, each factor depending on the set of variables $z_i$ is replaced by the corresponding moment variable. Thus, permuting the sets $z_1,\ldots,z_n$ only permutes the moment factors, leaving their product unchanged.
Therefore, for the symmetrization operator
$$
\operatorname{Sym}_nF
=
\frac1{n!}\sum_{\sigma\in S_n}
F(z_{\sigma(1)},\ldots,z_{\sigma(n)}),
$$
we have
$$
\Phi_n(F)=\Phi_n(\operatorname{Sym}_nF).
$$

For a moment monomial
$M=\prod_{i=1}^{n}m_{\alpha_i,\beta_i,r_i,s_i}$,
define its complete polarization by
$$
\operatorname{Pol}_n(M)
=
\frac1{n!}\sum_{\sigma\in S_n}
\prod_{i=1}^{n}
b_{\alpha_i,\beta_i,r_i,s_i}(z_{\sigma(i)}).
$$
Linear extension gives the map
$\operatorname{Pol}_n:\mathcal{DP}_n\to R_n^{S_n},$ where $R_n^{S_n}$ denotes the subspace of polynomials in $R_n$
invariant under permutations of the $n$ sets of variables
$z_1,\ldots,z_n$.

\begin{example}\label{FF}
For $F=x_1y_1p_1q_1\in R_n$, we have
$\Phi_n(F)=m_{1,1,1,1}m_{0,0,0,0}^{\,n-1}$, with
$$
\operatorname{Pol}_n
\left(m_{1,1,1,1}m_{0,0,0,0}^{\,n-1}\right)
=
\frac1n\sum_{i=1}^{n}x_iy_ip_iq_i.
$$
Thus, polarization recovers the symmetric preimage of a moment polynomial in the space of polynomials in
first-jet coordinates.
\end{example}

For each basis element $X_k$, $k=1,\ldots,8$,
we retain the notation
$\mathscr L_k=\mathscr L_{X_k}$ for the operator acting
on polynomials in the coordinates of a single first jet
$z=(x,y,p,q)$. We denote by
$\mathscr L_k^{(i)}$ the same operator
acting only on the variables
$z_i=(x_i,y_i,p_i,q_i)$, that is, the operator obtained from
$\mathscr L_k$ by the substitution
$
(x,y,p,q)\longmapsto(x_i,y_i,p_i,q_i).
$
Then the operator of the diagonal action on $n$ sets of variables
takes the form
$$
\mathscr L_k^{[n]}
=
\sum_{i=1}^{n}\mathscr L_k^{(i)},
\qquad k=1,\ldots,8.
$$

The maps $\Phi_n$ and $\operatorname{Pol}_n$
not only relate multipoint kernels to moment
polynomials but are also compatible with the action of the Lie algebra
$\mathfrak{sl}_3$, as formalized in the following lemma.
\begin{lemma}
\label{lem:polarization-identities}
For every $n\geq1$, the following identities hold:
\begin{align}
\Phi_n\circ\operatorname{Pol}_n
&=\operatorname{id}_{\mathcal{DP}_n},
\label{Phi-Pol-identity}\\
\operatorname{Pol}_n\circ\Phi_n
&=\operatorname{Sym}_n.
\label{Pol-Phi-sym}
\end{align}
Moreover, for every $k=1,\ldots,8$,
$F\in R_n$, and $I\in\mathcal{DP}_n$, we have
\begin{equation}
\label{polarization-equivariance}
\begin{aligned}
\mathscr D_k\Phi_n(F)
&=\Phi_n(\mathscr L_k^{[n]}F),\\
\operatorname{Pol}_n(\mathscr D_kI)
&=\mathscr L_k^{[n]}\operatorname{Pol}_n(I).
\end{aligned}
\end{equation}
In particular, $\Phi_n$ and $\operatorname{Pol}_n$ are mutually
inverse maps between $R_n^{S_n}$ and
$\mathcal{DP}_n$ and intertwine the corresponding actions
of the Lie algebra $\mathfrak{sl}_3$.
\end{lemma}

\begin{proof}
All maps in these identities are linear,
so it suffices to verify them on monomials.

Let
$b_i=b_{\alpha_i,\beta_i,r_i,s_i}$,
$i=1,\ldots,n$, be single-point moment
kernels of monomial form, and let
$$
M=\prod_{i=1}^{n}m_{b_i}
$$
be the corresponding moment monomial. By the definition of polarization,
$$
\operatorname{Pol}_n(M)
=
\frac1{n!}\sum_{\sigma\in S_n}
\prod_{i=1}^{n}b_i(z_{\sigma(i)}).
$$
For every permutation $\sigma\in S_n$, applying $\Phi_n$
replaces the factors $b_i(z_{\sigma(i)})$ with the corresponding
moments $m_{b_i}$, so
$$
\Phi_n\left(
\prod_{i=1}^{n}b_i(z_{\sigma(i)})
\right)
=
\prod_{i=1}^{n}m_{b_i}=M.
$$
Hence,
$$
\Phi_n(\operatorname{Pol}_n(M))
=
\frac1{n!}\sum_{\sigma\in S_n}M=M,
$$
which proves~\eqref{Phi-Pol-identity}.
Now let
$$
F=\prod_{i=1}^{n}b_i(z_i)\in R_n.
$$
Then $\Phi_n(F)=M$, and therefore
$$
\operatorname{Pol}_n(\Phi_n(F))
=
\frac1{n!}\sum_{\sigma\in S_n}
\prod_{i=1}^{n}b_i(z_{\sigma(i)})
=
\operatorname{Sym}_nF.
$$
By linearity, we obtain~\eqref{Pol-Phi-sym} for
arbitrary $F\in R_n$. In particular, if $F$ is symmetric,
then $\operatorname{Sym}_nF=F$, which implies that
$\Phi_n$ and $\operatorname{Pol}_n$ are mutually inverse on
$R_n^{S_n}$ and $\mathcal{DP}_n$.

We now verify the first identity
in~\eqref{polarization-equivariance}.
We have
$$
\mathscr L_k^{[n]}F
=
\sum_{i=1}^{n}
(\mathscr L_kb_i)(z_i)
\prod_{j\neq i}b_j(z_j).
$$
Applying $\Phi_n$ and using the definition of $\Phi_1$,
we obtain
$$
\Phi_n(\mathscr L_k^{[n]}F)
=
\sum_{i=1}^{n}
\Phi_1(\mathscr L_kb_i)
\prod_{j\neq i}m_{b_j}.
$$
On the other hand, by definition,
$
\mathscr D_km_{b_i}
=
\Phi_1(\mathscr L_kb_i),
$
and $\mathscr D_k$ acts on $\mathcal{DP}$ as a derivation.
Therefore,
$$
\mathscr D_k\Phi_n(F)
=
\mathscr D_k\left(\prod_{i=1}^{n}m_{b_i}\right)
=
\sum_{i=1}^{n}
\Phi_1(\mathscr L_kb_i)
\prod_{j\neq i}m_{b_j}
=
\Phi_n(\mathscr L_k^{[n]}F).
$$
Linearity gives the first equality
in~\eqref{polarization-equivariance} for arbitrary
$F\in R_n$.

For the second equality, observe that the operator
$\mathscr L_k^{[n]}$ commutes with permutations of the sets
of variables and hence with the operator $\operatorname{Sym}_n$.
Let $I\in\mathcal{DP}_n$ and
$
F=\operatorname{Pol}_n(I).
$
Then $F$ is symmetric and, by~\eqref{Phi-Pol-identity},
$\Phi_n(F)=I$. Thus,
\begin{align*}
\operatorname{Pol}_n(\mathscr D_kI)
&=
\operatorname{Pol}_n\bigl(\mathscr D_k\Phi_n(F)\bigr)
=
\operatorname{Pol}_n\bigl(\Phi_n(\mathscr L_k^{[n]}F)\bigr)
=
\operatorname{Sym}_n(\mathscr L_k^{[n]}F)
=
\mathscr L_k^{[n]}\operatorname{Sym}_nF=
\\
&=
\mathscr L_k^{[n]}F
=
\mathscr L_k^{[n]}\operatorname{Pol}_n(I).
\end{align*}
This proves the second equality and completes the proof of the lemma.

\end{proof}

\subsection{Completeness theorem}

We first establish a connection between the group law of relative
invariance and the conditions of annihilation by the corresponding
Lie algebra operators. For a polynomial kernel $F$, the weight $-1$ condition
is written as
$
\mathcal J_g(\boldsymbol z)F(g\boldsymbol z)=F(\boldsymbol z).
$
Differentiating this equality along the one-parameter subgroups
generated by the basis elements $X_k$ gives the system
\begin{equation}
\label{relative-equations}
\mathscr L_k^{[n]}F=0,
\qquad k=1,\ldots,8.
\end{equation}

Conversely, suppose that
\eqref{relative-equations} holds.
For the one-parameter subgroup
$g_t=\exp(tX_k)$, differentiating the coordinate
Jacobian multiplier $\mathcal J_{g_t}$
and applying the chain rule gives
$$
\frac{d}{dt}
\left[
\mathcal J_{g_t}(\boldsymbol z)F(g_t\boldsymbol z)
\right]
=
\mathcal J_{g_t}(\boldsymbol z)
\bigl(\mathscr L_k^{[n]}F\bigr)(g_t\boldsymbol z)
=
0.
$$
At $t=0$, the expression in square brackets equals
$F(\boldsymbol z)$, and hence
$
\mathcal J_{g_t}(\boldsymbol z)F(g_t\boldsymbol z)
=
F(\boldsymbol z)
$
for each of these one-parameter subgroups.
By the multiplication rule for Jacobians, the same equality
holds for their compositions. Since
$X_1,\ldots,X_8$ form a basis of
$\mathfrak{sl}_3$ and the connected group
$PGL(3,\mathbb R)\simeq SL(3,\mathbb R)$
is generated by its one-parameter subgroups,
the relative invariance law holds for all
$g\in PGL(3,\mathbb R)$.

Thus,
\begin{equation}
\label{relative-characterization}
F\in\mathcal P_{n,-1}
\quad\Longleftrightarrow\quad
\mathscr L_k^{[n]}F=0
\quad\text{for }k=1,\ldots,8.
\end{equation}

This characterization of relative invariants, together with
the intertwining property of the Lie algebra actions, now allows us
to establish an exact correspondence between multipoint kernels
of weight $-1$ and projective differential moment
invariants.

\begin{theorem}
\label{thm:moment-completeness-formal}
For every $n\geq 1 $, the map $\Phi_n$ sends
$\mathcal P_{n,-1}$ to $\mathcal{MI}_n$. Its restriction
$$
\Phi_n:
(\mathcal P_{n,-1})^{S_n}
\longrightarrow\mathcal{MI}_n
$$
is an isomorphism of vector spaces whose inverse
is the complete polarization operator $\operatorname{Pol}_n$.
\end{theorem}

\begin{proof}
First, let $F\in\mathcal P_{n,-1}$.
By~\eqref{relative-characterization},
we have
$
\mathscr L_k^{[n]}F=0,
k=1,\ldots,8.
$
By Lemma~\ref{lem:polarization-identities}, the map
$\Phi_n$ intertwines the corresponding Lie algebra actions, so
$$
\mathscr D_k\Phi_n(F)
=
\Phi_n(\mathscr L_k^{[n]}F)
=
0,
\qquad k=1,\ldots,8.
$$
Since $\Phi_n(F)\in\mathcal{DP}_n$, the definition
of a moment invariant implies that
$
\Phi_n(F)\in\mathcal{MI}_n.
$
Thus, in particular, the restriction of $\Phi_n$ to
$(\mathcal P_{n,-1})^{S_n}$ is well defined.

Now let $I\in\mathcal{MI}_n$ and
$
F=\operatorname{Pol}_n(I).
$
By the definition of polarization, $F$ is a symmetric polynomial.
Moreover,
$$
\mathscr D_kI=0,
\qquad k=1,\ldots,8.
$$
Using Lemma~\ref{lem:polarization-identities} again,
we obtain
$$
\mathscr L_k^{[n]}F
=
\mathscr L_k^{[n]}\operatorname{Pol}_n(I)
=
\operatorname{Pol}_n(\mathscr D_kI)
=
0,
\qquad k=1,\ldots,8.
$$
By~\eqref{relative-characterization},
this means that $F\in\mathcal P_{n,-1}$. Hence,
$$
\operatorname{Pol}_n(I)
\in
(\mathcal P_{n,-1})^{S_n}.
$$

It remains to verify that the maps are mutually inverse.
By Lemma~\ref{lem:polarization-identities}, for every
$I\in\mathcal{MI}_n$, we have
$
\Phi_n(\operatorname{Pol}_n(I))=I.
$
Therefore, $\Phi_n$ is surjective onto $\mathcal{MI}_n$.
For symmetric $F\in\mathcal P_{n,-1}$, the same lemma gives
$$
\operatorname{Pol}_n(\Phi_n(F))
=
\operatorname{Sym}_nF
=
F.
$$
Thus, $\Phi_n$ is also injective on
$(\mathcal P_{n,-1})^{S_n}$, and
$\operatorname{Pol}_n$ is its inverse.
\end{proof}

Thus, describing moment invariants reduces
to describing the symmetric component of $\mathcal P_{n,-1}$.
This approach is analogous to the symbolic method of classical
invariant theory: moment polynomials are represented
by compact multipoint moment kernels
that are relative invariants of weight $-1$.

The cases $n=1$ and $n=2$ require separate
consideration.
\begin{theorem}
\label{thm:low-point-relative-invariants}
For the diagonal projective action on one or two
first jets,
$$
\mathcal P_{1,-1}=0,
\qquad
\mathcal P_{2,-1}=0.
$$
\end{theorem}

\begin{proof}
By the characterization of relative invariants of weight $-1$
in terms of the Lie algebra action, every polynomial
$F\in\mathcal P_{n,-1}$ satisfies the system
$$
\mathscr L_k^{[n]}F=0,
\qquad k=1,\ldots,8.
$$

First, consider the case $n=1$. Let
$
F=F(x,y,p,q)\in\mathcal P_{1,-1}.
$
For the translation generators
$
X_1=\partial_x,
X_2=\partial_y,
$
we have
$
\operatorname{div}X_1=\operatorname{div}X_2=0,
$
and their first prolongations coincide with the generators
themselves. Therefore, the corresponding operators on the space
of a single first jet take the form
$$
\mathscr L_1^{[1]}=\partial_x,
\qquad
\mathscr L_2^{[1]}=\partial_y.
$$
The conditions
$$
\mathscr L_1^{[1]}F=0,
\qquad
\mathscr L_2^{[1]}F=0
$$
give the system
$$
\frac{\partial F}{\partial x}=0,
\qquad
\frac{\partial F}{\partial y}=0.
$$
Thus, $F$ does not depend on $x$ and $y$, that is,
$
F=F(p,q).
$

For the generators
$
X_3=x\partial_x,
X_4=y\partial_y,
$
we have
$
\operatorname{div}X_3=1,
\operatorname{div}X_4=1,
$
and their first prolongations take the form
$$
X_3^{(1)}
=
x\partial_x-p\partial_p,
\qquad
X_4^{(1)}
=
y\partial_y-q\partial_q.
$$
Therefore, the corresponding operators on the space of a single first
jet are
$$
\mathscr L_3^{[1]}
=
x\partial_x-p\partial_p+1,
\qquad
\mathscr L_4^{[1]}
=
y\partial_y-q\partial_q+1.
$$
Since we have already established that $F=F(p,q)$, the conditions
$$
\mathscr L_3^{[1]}F=0,
\qquad
\mathscr L_4^{[1]}F=0
$$
take the form
$$
-p\frac{\partial F}{\partial p}+F=0,
\qquad
-q\frac{\partial F}{\partial q}+F=0,
$$
or, equivalently,
$$
p\frac{\partial F}{\partial p}=F,
\qquad
q\frac{\partial F}{\partial q}=F.
$$
The first equality means that $F$ is homogeneous of degree $1$
in the variable $p$, while the second means that it is homogeneous of degree $1$
in the variable $q$. Since $F$ is a polynomial, it follows that
$$
F=c\,pq
$$
for some constant $c$.

Now, for the generator
$
X_5=y\partial_x,
$
we have
$
\mathscr L_5^{[1]}
=
y\partial_x-p\partial_q.
$
Hence,
$$
0=\mathscr L_5^{[1]} F
=-p\frac{\partial}{\partial q}(cpq)
=-cp^2,
$$
so $c=0$. Thus,
$
\mathcal P_{1,-1}=0.
$

We now turn to the case $n=2$. Let
$
F=
F(x_1,y_1,p_1,q_1,x_2,y_2,p_2,q_2)
\in\mathcal P_{2,-1}.
$
As in the case of a single jet, we first use
the translation generators. For these,
$$
\mathscr L_1^{[2]}
=
\partial_{x_1}+\partial_{x_2},
\qquad
\mathscr L_2^{[2]}
=
\partial_{y_1}+\partial_{y_2}.
$$
Therefore, the equations
$$
\mathscr L_1^{[2]}F=0,
\qquad
\mathscr L_2^{[2]}F=0
$$
mean that $F$ depends on the coordinates of the two points
only through their differences
$$
u=x_2-x_1,
\qquad
v=y_2-y_1.
$$
Thus,
$
F=F(u,v,p_1,q_1,p_2,q_2).
$

We now show that $F$ vanishes on the hyperplane
$v=0$.
For the generator
$
X_8=xy\partial_x+y^2\partial_y,
$
we have the following first prolongation:
$$
\mathscr L_8
=
xy\partial_x+y^2\partial_y
-yp\partial_p
-(xp+2yq)\partial_q
+3y.
$$
Thus, on two jets,
$$
\mathscr L_8^{[2]}
=
\sum_{i=1}^2
\left(
x_i y_i\partial_{x_i}
+y_i^2\partial_{y_i}
-y_i p_i\partial_{p_i}
-(x_i p_i+2y_iq_i)\partial_{q_i}
+3y_i
\right).
$$
Since $F$ depends only on coordinate differences,
at $v=0$ we may set
$$
(x_1,y_1)=(0,0),
\qquad
(x_2,y_2)=(u,0).
$$
Therefore, only one term remains in $\mathscr L_8^{[2]}$:
$
-(x_2p_2)\partial_{q_2}
=-up_2\partial_{q_2}.
$
Thus,
$$
\left.\mathscr L_8^{[2]}\right|_{v=0}
=
-up_2\partial_{q_2}.
$$
Hence, the equation $\mathscr L_8^{[2]}F=0$ implies
$$
-u p_2
\left.
\frac{\partial F}{\partial q_2}
\right|_{v=0}
=0,
$$
and therefore, as a polynomial identity,
$$
\left.
\frac{\partial F}{\partial q_2}
\right|_{v=0}=0.
$$

For the generator $X_5=y\partial_x$, we have
$
\mathscr L_5^{[2]}
=
v\partial_u
-p_1\partial_{q_1}
-p_2\partial_{q_2}.
$
At $v=0$, the equation $\mathscr L_5^{[2]}F=0$ takes the form
$$
-p_1
\left.
\frac{\partial F}{\partial q_1}
\right|_{v=0}
-p_2
\left.
\frac{\partial F}{\partial q_2}
\right|_{v=0}
=0.
$$
Using the equality already established,
$F_{q_2}|_{v=0}=0$, we obtain
$$
\left.
\frac{\partial F}{\partial q_1}
\right|_{v=0}=0.
$$

Finally, for $X_4=y\partial_y$, we have
$
\mathscr L_4^{[2]}
=
v\partial_v
-q_1\partial_{q_1}
-q_2\partial_{q_2}
+2.
$
Thus, at $v=0$,
$$
-q_1
\left.
\frac{\partial F}{\partial q_1}
\right|_{v=0}
-q_2
\left.
\frac{\partial F}{\partial q_2}
\right|_{v=0}
+2F|_{v=0}
=0.
$$
The first two terms vanish, giving
$$
F|_{v=0}=0.
$$
Hence, $F$ is divisible by $v$, and
$
F=vG_1
$
for some polynomial $G_1$.

For the generator
$
X_6=x\partial_y,
$
the corresponding operator on two jets takes the form
$$
\mathscr L_6^{[2]}
=
u\partial_v
-q_1\partial_{p_1}
-q_2\partial_{p_2}.
$$
Substituting $F=vG_1$ into the equation
$\mathscr L_6^{[2]}F=0$ gives
$$
u\left(G_1+v\frac{\partial G_1}{\partial v}\right)
-vq_1\frac{\partial G_1}{\partial p_1}
-vq_2\frac{\partial G_1}{\partial p_2}
=0.
$$
Setting $v=0$, we obtain
$
u\,G_1|_{v=0}=0.
$
Since this is a polynomial identity, it follows that
$$
G_1|_{v=0}=0.
$$
Thus, $G_1$ is also divisible by $v$, and hence
$
F=v^2G_2
$
for some polynomial $G_2$.

The same argument can be repeated. Indeed, suppose that
for some $m\geq1,$
$
F=v^mG_m.
$
Then the equation $\mathscr L_6^{[2]}F=0$ takes the form
$$
v^{m-1}
\left[
mu\,G_m
+
v\left(
u\frac{\partial G_m}{\partial v}
-q_1\frac{\partial G_m}{\partial p_1}
-q_2\frac{\partial G_m}{\partial p_2}
\right)
\right]
=0.
$$
Since the polynomial ring has no zero divisors,
the expression in square brackets vanishes identically.
At $v=0$, this gives
$
mu\,G_m|_{v=0}=0,
$
and therefore
$
G_m|_{v=0}=0.
$
Thus, $G_m$ is divisible by $v$, that is, $F$ is divisible by
$v^{m+1}$.

Consequently, by induction, $F$ is divisible by $v^m$ for
every $m\geq1$. Since $F$ is a polynomial of finite
degree in $v$, this is possible only if
$
F=0.
$
Hence,
$
\mathcal P_{2,-1}=0.
$
Together with the case $n=1$ already proved, this completes
the proof of the theorem.
\end{proof}

\subsection{Derivative degree and integral identities}

The structural description of admissible monomials gives an important restriction
on the moments that occur in the invariants.

\begin{lemma}
\label{minimum-derivative-degree}
Let
$\mathcal{DP}_{\geq3}
=\mathbb R[m_{\alpha,\beta,r,s}\mid r+s\geq3]$.
For every $n\geq3$, we have
$\mathcal{MI}_n\subset\mathcal{DP}_{\geq3}$.
\end{lemma}

\begin{proof}
For an admissible monomial~\eqref{Pba}, the degree in the derivatives
$p_i,q_i$ equals
$$
d_{\boldsymbol b}(i)+\sum_{j\neq i}a_{ij}
=
3+\sum_{j\neq i}a_{ji}\geq3,
$$
where~\eqref{admissibility_conditions} has been used.
Thus, each moment factor in its image under
$\Phi_n$ has $r+s\geq3$. The assertion follows from the description
of $\mathcal P_{n,-1}$ by admissible monomials and
Theorem~\ref{thm:moment-completeness-formal}.
\end{proof}

On the entire algebra $\mathcal{DP}$, the integral realization~\eqref{Phi-n} is not
injective. For example,
we have $$m_{0,0,1,0}(f)=\frac12\iint_{\mathbb{R}^2} \partial_x(f^2)\,dx\,dy=0.$$
However, as the following theorem shows, there are no universal
polynomial identities on the subalgebra from
Lemma~\ref{minimum-derivative-degree}.

\begin{theorem}
\label{moment-functional-independence}
Let $P_1,\ldots,P_N\in\mathbb R[x,y,p,q]$ be linearly
independent, and suppose that each of their monomials has degree in $p,q$
at least two. Then the functionals
$$
M_i(f)
=
\iint_{\mathbb R^2}
f(x,y)P_i(x,y,p,q)\,dx\,dy,
\qquad i=1,\ldots,N,
$$
are algebraically independent on the class of twice-differentiable
functions $f$ with compact support.
In particular, the integral realization of the algebra
$\mathcal{DP}_{\geq3}$ is injective.
\end{theorem}
\begin{proof}
For the proof, we use standard methods of the calculus of variations for functionals depending on the function $f$ and its first derivatives.
Let $E_i$ denote
$$
E_i
=
P_i
-
D_x\!\left(f\frac{\partial P_i}{\partial p}\right)
-
D_y\!\left(f\frac{\partial P_i}{\partial q}\right).
$$
the Euler--Lagrange expression for the integrand $fP_i$.
Then, after simplification and integration by parts, the first variation of the functional $M_i$ takes the form
$$
DM_i(f)[h]
=
\int_{\mathbb R^2}
E_i(j^2f)\,h\,dx\,dy,
$$
where $j^2f=(x,y,f,p,q,f_{xx},f_{xy},f_{yy})$
is the second jet of the function $f$.

We first prove that $E_1,\ldots,E_N$ are linearly independent
as functions on the second-jet space.
Suppose that $\sum_i c_iE_i=0$, and let $P=\sum_i c_iP_i$.
Then the Euler--Lagrange expression for $fP$ vanishes identically.
Its coefficients of $f_{xx},f_{xy},f_{yy}$ are
$-fP_{pp}$, $-2fP_{pq}$, and $-fP_{qq}$, respectively.
Since the equality holds for all jet values,
we have $P_{pp}=P_{pq}=P_{qq}=0$. Thus, $P$ has degree
at most one in the variables $p,q$.
However, all monomials of $P$ have degree in $p,q$
at least two, so $P=0$.
The linear independence of the $P_i$ now gives
$c_1=\cdots=c_N=0$.

Thus, the vectors of values
$(E_1(\theta),\ldots,E_N(\theta))$ span
$\mathbb R^N$. Choose jets $\theta_1,\ldots,\theta_N$
such that the matrix
$$
A=\bigl(E_i(\theta_j)\bigr)_{i,j=1}^{N}
$$
is nonsingular.
Since the determinant of $A$
depends continuously on the jet coordinates and is nonzero,
it remains nonzero under sufficiently small perturbations.
Therefore, the base points $(x_j,y_j)$ can be chosen to be pairwise distinct
without affecting the nonsingularity of the matrix.

These jets can be realized by a single twice-differentiable function
$f_*$.
Indeed, near each base point $(x_j,y_j)$, choose
a quadratic polynomial whose second jet at that point
coincides with the prescribed jet $\theta_j$.
Next, multiply this polynomial by a smooth function $\chi_j$
that equals $1$ in a neighborhood of $(x_j,y_j)$
and equals $0$ outside a slightly larger neighborhood.
Multiplication by $\chi_j$ therefore does not change the value of the polynomial
or its derivatives up to second order at $(x_j,y_j)$.
Since the base points are pairwise distinct, these neighborhoods can
be chosen to be pairwise disjoint. The sum of the resulting localized
polynomials gives a compactly supported function $f_*$ such that
$$
j^2f_*(x_j,y_j)=\theta_j,
\qquad j=1,\ldots,N.
$$

Now choose nonnegative twice-differentiable functions
$h_j$ with unit integral
and supports in small neighborhoods of the corresponding points.
By the continuity of $E_i(j^2f_*)$, as these
supports shrink, we have
$$
DM_i(f_*)[h_j]
=
\int_{\mathbb R^2}E_i(j^2f_*)h_j\,dx\,dy
\longrightarrow E_i(\theta_j).
$$
Thus, the supports can be chosen so that the matrix
$B=(DM_i(f_*)[h_j])_{i,j=1}^{N}$ is nonsingular.

Consider the map
$$
\Psi:\mathbb R^N\longrightarrow\mathbb R^N,
\qquad
\Psi(t)
=
\left(
M_i\left(f_*+\sum_{j=1}^{N}t_jh_j\right)
\right)_{i=1}^{N}.
$$
It is polynomial in $t$, and its Jacobian matrix
at zero equals $B$. By the inverse function theorem,
the image of $\Psi$ contains an open subset of $\mathbb R^N$.

If there were a polynomial $R$ such that
$R(M_1(f),\ldots,M_N(f))=0$ for all admissible $f$,
it would vanish on this open subset.
Therefore, $R$ would be the zero polynomial, proving
the algebraic independence of the functionals.

Finally, every polynomial in $\mathcal{DP}_{\geq3}$
depends on only finitely many moment
generators. Their distinct monomial kernels
$x^\alpha y^\beta p^rq^s$, $r+s\geq3$, satisfy
the hypotheses of the assertion just proved. Thus, no nonzero
element of $\mathcal{DP}_{\geq3}$ is realized
as the zero functional, that is, the integral realization
of this algebra is injective.
\end{proof}

Thus, the passage from formal moment variables to their
integral definitions does not produce additional
algebraic relations among the constructed
invariants.

\section{Low-Order Projective Differential Moment Invariants}
\label{low_order_invariants}

Theorem~\ref{thm:moment-completeness-formal} reduces the construction
of moment invariants to the description of symmetric relative invariants of weight $-1$.
In this section, we consider explicit constructions of moment
degrees $3$ and $4$.

\subsection{Moment invariants of degrees $1$ and $2$}

Before turning to the first nontrivial constructions
of moment degree $3$, we consider the two lower degrees.
Theorem~\ref{thm:low-point-relative-invariants} shows
that the corresponding spaces of relative invariants of weight $-1$
for one and two first jets are zero, which immediately
leads to the following assertion.
\begin{theorem}
There are no nonzero projective differential moment
invariants of moment degrees $1$ and $2$:
$$
\mathcal{MI}_1=\mathcal{MI}_2=\{0\}.
$$
\end{theorem}
\begin{proof}
Let $I\in\mathcal{MI}_n$, where $n=1$ or $n=2$.
By the property of polarization that intertwines the corresponding
Lie algebra actions,
$$
\operatorname{Pol}_n(\mathscr D_k I)
=
\mathscr L_k^{[n]}\operatorname{Pol}_n(I),
\qquad k=1,\ldots,8.
$$
Since $I$ is a moment invariant,
$\mathscr D_k I=0$ for all $k$, and hence
$$
\mathscr L_k^{[n]}\operatorname{Pol}_n(I)=0,
\qquad k=1,\ldots,8.
$$
Thus,
$
\operatorname{Pol}_n(I)\in\mathcal P_{n,-1}.
$
By Theorem~\ref{thm:low-point-relative-invariants},
$
\mathcal P_{1,-1}=\mathcal P_{2,-1}=0,
$
so
$
\operatorname{Pol}_n(I)=0.
$
Finally, the identity
$
\Phi_n\circ\operatorname{Pol}_n
=
\operatorname{id}_{\mathcal{DP}_n}
$
gives
$
I
=
\Phi_n(\operatorname{Pol}_n(I))
=
0.
$
Therefore, there are no nonzero moment invariants of degrees $1$ and $2$.
\end{proof}

\subsection{Adapted $D$-moments}

We use the $D$-moments
$\phi_{\alpha\beta rst}$ introduced
in Subsection~\ref{subsec:degree3_results}.
The values of the $D$-order and the numbers
of monomials given below characterize particular
$D$-moment representations of the invariants.
We also use the spaces $W_a$,
$a\geq0$, defined
in Subsection~\ref{subsec:degree3_results}.

We regard the space $W_a$ as a subspace of
$\mathcal{DP}_1$, identifying each symbol
$\phi_{\alpha\beta rst}$ with its expansion
in the moments $m_{\alpha,\beta,r,s}$.
At the same time, the bound
$$
\operatorname{ord}_D\phi_{\alpha\beta rst}
=
\alpha+\beta+r+s+t
\leq a+(a+3)=2a+3
$$
applies specifically to this $D$-moment notation.
After expanding $w=xp+yq$, the order of the moments
$m_{\alpha,\beta,r,s}$ in the corresponding expansion
may be higher.

The vector space $W_0$ is spanned by the ten moments
$\phi_{00ijk}$, $i+j+k=3$.
The coefficients of the ternary cubic from
Section~\ref{sec:main_results} are related to them by
$$
\mu_{ijk}=(-1)^k\phi_{00ijk}.
$$
Thus, both systems of moments span the same
space $W_0$.

\subsection{Moment degree three} \label{subsec:degree_three}

We first prove Theorem~\ref{all_degree3}, which gives
a general description of invariants of moment degree $3$.

\begin{proof}[Proof of Theorem~\ref{all_degree3}]
By the theorem on the structure of the space
$\mathcal P_{n,-1}$, each of its elements is a linear
combination of admissible monomials
$
P_{\boldsymbol b,\boldsymbol a}.
$
For $n=3$, the set $\mathcal T_3$ contains only the triple
$\{1,2,3\}$. Therefore, the first admissibility condition
$
\sum_{I\in\mathcal T_3}b_I=3
$
gives
$
b_{123}=3.
$
Thus, every admissible monomial contains the factor
$D_{123}^{\,3}$. Moreover,
$$
d_{\boldsymbol b}(1)
=
d_{\boldsymbol b}(2)
=
d_{\boldsymbol b}(3)
=3,
$$
so the second admissibility condition takes the form
$
\partial_{\boldsymbol a}(i)=0,
\qquad i=1,2,3.
$
This means that, for each index, the total number
of its occurrences as the first index in the factors $C_{ij}$
equals the total number of its occurrences as the second index.
Therefore, the corresponding monomial in $C_{ij}$ decomposes
into a product of closed cycles. For three indices, these
cycles are
$
\mathsf X,
\mathsf Y,
\mathsf Z,U,V;
$
see~\eqref{dva} and~\eqref{try}.
Thus, every element of $\mathcal P_{3,-1}$ has the form
\begin{equation}
\label{DD}
D_{123}^{\,3}Q(\mathsf X,\mathsf Y,\mathsf Z,U,V)
\end{equation}
for some polynomial $Q$.

It remains to account for symmetry under permutations
of the three points. Interchanging two indices changes the sign of the determinant
$D_{123}$, and hence also changes the sign of $D_{123}^{\,3}$.
Therefore, the product~\eqref{DD}
is symmetric under permutations of the points if and only
if $Q$ changes sign under every odd
permutation, that is, if it is antisymmetric.

By Theorem~\ref{thm:moment-completeness-formal},
the map $\Phi_3$ establishes a one-to-one
correspondence between symmetric elements of
$\mathcal P_{3,-1}$ and moment invariants in
$\mathcal{MI}_3$. Therefore, every $I\in\mathcal{MI}_3$ has the form
$$
I=\Phi_3\left(
D_{123}^{\,3}Q(\mathsf X,\mathsf Y,\mathsf Z,U,V)
\right),
$$
where $Q$ is antisymmetric, and every such expression
is a moment invariant.
\end{proof}

We also prove Theorem~\ref{J-family}. To this end,
we specify which of the admissible monomials described above
correspond to the space $\operatorname{Sym}^3(W_a)$.

\begin{proof}[Proof of Theorem~\ref{J-family}]
To count degrees, we use the formal expression
$$
C_{ij}=p_ix_j+q_iy_j-w_i,
\qquad
\ell_i=(p_i,q_i,-w_i),
$$
in which $w_i$ is temporarily an independent symbol.
Each factor $C_{ij}$ is homogeneous of degree $1$
in the components of $\ell_i$ and has degree
at most $1$ in the coordinates $(x_j,y_j)$.
The coordinate degree bounds below
refer to this formal expression.
After substituting $w_i=x_ip_i+y_iq_i$,
the degree in the components of $\ell_i$ becomes
the same total degree in the derivatives $p_i,q_i$,
while the coordinate degrees may increase.

Consider an admissible monomial
$$
D_{123}^{\,3}\prod_{i\ne j}C_{ij}^{b_{ij}}.
$$
For each $i=1,2,3$, set
$$
r_i=\sum_{j\ne i}b_{ij},
\qquad
s_i=\sum_{j\ne i}b_{ji}.
$$
Here, $r_i$ is the total number of occurrences of the index $i$
as the first index in the factors $C_{ij}$, while $s_i$ is
the total number of its occurrences as the second index.
As shown above, the admissibility condition for $n=3$
gives
$
r_i=s_i,
\qquad i=1,2,3.
$

The factor $D_{123}^{\,3}$ has degree $3$ in
the components of each covector $\ell_i$.
Therefore, at point $i$, the degree of the entire monomial
in the components of $\ell_i$ equals
$
3+r_i,
$
while its degree in the coordinates $(x_i,y_i)$ does not exceed
$s_i$. Each operator $\mathcal L_k^{(i)}$ preserves
the total degree in $(p_i,q_i)$, so the space
$\mathcal P_{3,-1}$ decomposes according to the tridegree
in the derivatives at the three points. The map $\Phi_3$
sends the symmetric components of this decomposition
to components of $\mathcal{DP}_3$ determined by
the multisets of derivative degrees
of the three moment factors. These components
form a direct sum.

Since all elements of $W_a$ have derivative degree
$a+3$, a symmetric kernel whose
image belongs to $\operatorname{Sym}^3(W_a)$
can correspond only to the tridegree
$(a+3,a+3,a+3)$. Thus, in its expansion
in admissible monomials, it suffices to consider
only monomials for which
$$
r_i=s_i=a,\qquad i=1,2,3.
$$
Conversely, for such monomials, the formal
coordinate degree at each point
does not exceed $a$, and the degree in the components
of $\ell_i$ equals $a+3$, so their images
belong to $\operatorname{Sym}^3(W_a)$.

Consider the monomial
$$
\mathsf X^{\rho_1}
\mathsf Y^{\rho_2}
\mathsf Z^{\rho_3}U^uV^v.
$$
The conditions $s_1=s_2=s_3=a$ give
$$
\rho_1+\rho_2+u+v
=
\rho_1+\rho_3+u+v
=
\rho_2+\rho_3+u+v
=a.
$$
Hence,
$$
\rho_1=\rho_2=\rho_3=\rho.
$$
Using the relation
$\mathsf X\mathsf Y\mathsf Z=UV$, we obtain
$$
\mathsf X^\rho\mathsf Y^\rho\mathsf Z^\rho U^uV^v
=
U^{u+\rho}V^{v+\rho}.
$$
Moreover,
$$
(u+\rho)+(v+\rho)=a.
$$
Thus, all such monomials have the form
$$
U^{a-k}V^k,
\qquad k=0,\ldots,a.
$$

Under an odd permutation of the three points, the quantities $U$ and $V$
are interchanged, while $D_{123}^{\,3}$ changes sign.
Therefore, symmetric kernels correspond to the antisymmetric part
of the space of homogeneous polynomials of degree $a$ in $U,V$.
A basis of this part is given by
$$
U^{a-k}V^k-U^kV^{a-k},
\qquad
k=0,\ldots,
\left\lfloor\frac{a-1}{2}\right\rfloor.
$$
Applying $\Phi_3$ and Theorem~
\ref{thm:moment-completeness-formal}, we obtain a basis
of the corresponding space of moment invariants:
$$
J_{a,k}
=
\Phi_3\left[
D_{123}^{\,3}
\left(
U^{a-k}V^k-U^kV^{a-k}
\right)
\right].
$$
\end{proof}

The constructed family of invariants $J_{a,k}$ allows us
to establish an important structural property
of the algebra of moment invariants: it has no
finite generating system.

\begin{theorem}
\label{thm:nonfinite-generation}
The space $\mathcal{MI}_3$ is infinite-dimensional.
The algebra $\mathcal{MI}$ is not finitely generated
as an $\mathbb R$-algebra.
\end{theorem}

\begin{proof}
Assign to the moment generator
$m_{\alpha,\beta,r,s}$ the auxiliary degree $r+s$,
and extend it to moment monomials by adding
the degrees of their factors.

By Theorem~\ref{J-family}, for every $a\geq1$,
the invariant $J_{a,0}$ is nonzero and belongs to
$\mathcal{MI}_3$.
Each of its moment factors has derivative degree
$a+3$, so $J_{a,0}$ is homogeneous
of auxiliary degree $3(a+3)$.
These degrees are pairwise distinct, so the invariants
$J_{a,0}$, $a\geq1$, are linearly independent.
Thus,
$$
\dim_{\mathbb R}\mathcal{MI}_3=\infty.
$$

Suppose that $\mathcal{MI}$ is finitely generated.
Since the operators $\mathscr D_k$ preserve
moment degree, each homogeneous component
of a moment invariant is also an invariant.
Replacing a finite generating system
with the homogeneous components of its generators gives
a finite system of homogeneous generators.

However, $\mathcal{MI}_0=\mathbb R$ and
$\mathcal{MI}_1=\mathcal{MI}_2=0$.
Therefore, all nonzero generators of positive degree
have moment degree at least $3$.
Products of two or more such factors
have degree at least $6$.
Thus, $\mathcal{MI}_3$ would have to be the linear
span of finitely many generators
of degree $3$, contradicting its
infinite dimensionality.
\end{proof}

We next consider several explicit examples of the simplest invariants of degree $3$
:
\begin{align*}
I^{(3)}_5
&=
\Phi_3\left(D_{123}^{\,3}(U-V)\right)
=
J_{1,0},
\\
I^{(3)}_{7,1}
&=
\Phi_3\left(D_{123}^{\,3}(U-V)(\mathsf X+\mathsf Y+\mathsf Z)\right),
\\
I^{(3)}_{7,2}
&=
\Phi_3\left(D_{123}^{\,3}(U^2-V^2)\right).
\end{align*}
Table~\ref{n3-complexity} gives the number
of monomials after expanding the kernels and collecting like terms,
as well as the number of distinct $D$-moments used.
No additional linear relations between $D$-moments
are applied in this process. Thus, the reported numbers
characterize these particular representations rather than
the minimum possible complexity of the invariants.

\begin{table}[t]
\centering
\caption{Complexity of the $D$-moment representations
of invariants of moment degree $3$.}
\label{n3-complexity}
\begin{tabular}{c|c|r|r}
\hline
Invariant & $D$-order & Monomials & Distinct $D$-moments
\\ \hline
$I^{(3)}_5$     &5&342 &45\\
$I^{(3)}_{7,1}$ &7&5082&171\\
$I^{(3)}_{7,2}$ &7&3272&126\\
\hline
\end{tabular}
\end{table}

\subsection{Moment degree four}

For moment degree $4$, the basic construction is the kernel
\begin{equation}
\label{P0}
P_0=D_{123}D_{124}D_{134}D_{234}.
\end{equation}
Each index occurs in exactly three determinants,
so, by the conditions~\eqref{admissibility_conditions},
$P_0$ has relative weight $-1$.
Moreover, each adjacent transposition changes the signs
of two determinants and interchanges the others.
Thus, $P_0$ is symmetric in the point labels.

The corresponding moment invariant takes the form
$$
I^{(4)}_3
=
\Phi_4(P_0)
=
24S,
$$
where $S$ is normalized according to
the polynomial~\eqref{eq:Aronhold_S}.
Its displayed $D$-moment representation has order $3$
and contains $25$ monomials in $10$ distinct moments.

The determinant representations of the invariants $S$ and $T$
are given in Subsection~\ref{subsec:degree3_results}.
Both use moments from $W_0$,
but have different moment degrees:
$4$ for $S$ and $6$ for $T$.

New examples of invariants of $D$-order $5$ can be obtained by multiplying $P_0$
by symmetric polynomials in $C_{ij}$ that are unchanged
under projective transformations. Consider
\begin{align*}
E_2
&=
\sum_{1\leq i<j\leq4}C_{ij}C_{ji},
\\
E_3
&=
\sum_{1\leq i<j<k\leq4}
\left(
C_{ij}C_{jk}C_{ki}
+
C_{ik}C_{kj}C_{ji}
\right),
\\
E_4
&=
\sum_{(i,j,k)\in\operatorname{Perm}(2,3,4)}
C_{1i}C_{ij}C_{jk}C_{k1},
\end{align*}
where $\operatorname{Perm}(2,3,4)$ denotes the six
permutations of the triple $(2,3,4)$.

All of $E_2,E_3,E_4$ are symmetric absolute
invariants. Accordingly, the kernels $P_0E_2$, $P_0E_3$,
and $P_0E_4$ remain symmetric and retain
relative weight $-1$.

We obtain three moment invariants
$$
I^{(4)}_{5,1}
=\Phi_4(P_0E_2),
\quad
I^{(4)}_{5,2}
=\Phi_4(P_0E_3),
\quad
I^{(4)}_{5,3}
=\Phi_4(P_0E_4).
$$
Table~\ref{n4-complexity} gives the number
of monomials after expanding these kernels and collecting
like terms, as well as the number of distinct
$D$-moments used.

\begin{table}[t]
\centering
\caption{Complexity of the $D$-moment representations
of invariants of moment degree $4$.}
\label{n4-complexity}
\begin{tabular}{c|c|r|r}
\hline
Invariant & $D$-order & Monomials & Distinct $D$-moments
\\ \hline
$I^{(4)}_3$     &3&25  &10\\
$I^{(4)}_{5,1}$ &5&1014&55\\
$I^{(4)}_{5,2}$ &5&2858&55\\
$I^{(4)}_{5,3}$ &5&3318&45\\
\hline
\end{tabular}
\end{table}

These numbers show that, even at a fixed
moment degree, increasing the $D$-order substantially
increases the complexity of the moment expansion.

\subsection{What about rational differential moment invariants?}

The action of the Lie algebra $\mathfrak{sl}_3$ on the moment algebra
$\mathcal{DP}$ naturally extends to its field of fractions
$
\mathcal K=\operatorname{Frac}(\mathcal{DP}),
$
and let $\mathcal K^{\mathfrak{sl}_3}$ denote the field of $\mathfrak{sl}_3$-invariants.
It turns out that, in this case, rational invariants do not form
a fundamentally broader class than polynomial invariants.

\begin{theorem}
The following equality holds:
$$
\mathcal K^{\mathfrak{sl}_3}
=
\operatorname{Frac}(\mathcal{MI}).
$$
Thus, every rational projective differential
moment invariant can be expressed as a ratio of two
polynomial moment invariants.
\end{theorem}
This assertion is analogous to the standard result
of classical invariant theory concerning the field of rational
invariants of a factorial algebra under the action of a semisimple
algebraic group~\cite{Springer1998,PopovVinberg1994}.
Since, in our case, $\mathcal{DP}$ has an infinite
set of generators and invariance is defined
directly through the action of the Lie algebra $\mathfrak{sl}_3$,
we give a short direct proof.

\begin{proof}
Let
$$
R=\frac{P}{Q}\in\mathcal K^{\mathfrak{sl}_3},
\qquad
P,Q\in\mathcal{DP},\qquad
\gcd(P,Q)=1.
$$
For every $X\in\mathfrak{sl}_3$, we have
$
\mathscr D_XR=0,
$
and hence
$
Q\,\mathscr D_XP-P\,\mathscr D_XQ=0.
$
Since $P$ and $Q$ are relatively prime, this equality
implies
$$
P\mid\mathscr D_XP,
\qquad
Q\mid\mathscr D_XQ.
$$
Thus, for some $h_X\in\mathcal{DP}$, we have
$$
\mathscr D_XP=h_XP,
\qquad
\mathscr D_XQ=h_XQ.
$$
Here, we may assume that $P,Q\neq0$, since the case
$R=0$ is immediate.
Let $\deg_m$ denote the maximum total degree
of a polynomial in the moment variables.
The operator $\mathscr D_X$ preserves each homogeneous
component, so
$\deg_m(\mathscr D_XP)\leq\deg_mP$.
If $h_X\neq0$, then
$$
\deg_m(h_XP)=\deg_mh_X+\deg_mP,
$$
and hence $\deg_mh_X=0$.
Thus, $h_X=\lambda_X\in\mathbb R$;
homogeneity of $P$ and $Q$ is not required.

The map $X\mapsto\lambda_X$ is linear.
Under the standard convention for the commutator of vector
fields, we have
$$
[X^{(1)},Y^{(1)}]=[X,Y]^{(1)},
\qquad
\operatorname{div}[X,Y]
=
X(\operatorname{div}Y)-Y(\operatorname{div}X).
$$
Therefore, the weighted operators
$\mathscr L_X=X^{(1)}+\operatorname{div}X$
satisfy
$[\mathscr L_X,\mathscr L_Y]=\mathscr L_{[X,Y]}$.
By the definition of the action on moment generators, this gives
$
[\mathscr D_X,\mathscr D_Y]
=
\mathscr D_{[X,Y]}.
$
Applying this identity to $P$, we obtain
$$
\lambda_{[X,Y]}P
=
[\mathscr D_X,\mathscr D_Y]P
=
(\lambda_Y\lambda_X-\lambda_X\lambda_Y)P
=
0.
$$
Since $P\neq0$, we have $\lambda_{[X,Y]}=0$.
Since
$
\mathfrak{sl}_3=[\mathfrak{sl}_3,\mathfrak{sl}_3],
$
and the character $\lambda$ vanishes on all commutators,
$\lambda([X,Y])=0$, it vanishes identically on the entire
algebra $\mathfrak{sl}_3$. Hence,
$
\lambda_X=0$
{for all } $X\in\mathfrak{sl}_3.$
Thus,
$
\mathscr D_XP=\mathscr D_XQ=0,
$
that is, $P,Q\in\mathcal{MI}$. Therefore,
$$
R=\frac{P}{Q}
\in\operatorname{Frac}(\mathcal{MI}).
$$
The reverse inclusion is immediate.
\end{proof}

Thus, under the definition of projective invariance
adopted in this paper through the Lie algebra action, passing
to rational functions produces no new invariants:
the field of rational invariants is the field of fractions of the algebra
of polynomial invariants.

This situation differs fundamentally from the case of the full
affine group, which has nontrivial characters. Therefore,
for affine moments, there may exist rational
invariants whose numerator and denominator are not individually
polynomial absolute invariants. This is precisely the
mechanism underlying the classical rational affine
moment invariants of Flusser--Suk~\cite{FlusserSuk1993}.
For the action of $\mathfrak{sl}_3$, this mechanism is impossible, since
$
\mathfrak{sl}_3=[\mathfrak{sl}_3,\mathfrak{sl}_3],
$
and hence this Lie algebra has no nontrivial characters.

\section*{Code availability}

The implementation is available at

\url{https://github.com/LeonidBed/projective-differential-moment-invariants}.

The repository provides routines for computing $S$, $T$,
and the degree-three invariants denoted in the code by
\texttt{I3\_cycle}, \texttt{I3\_cycle\_edge}, and
\texttt{I3\_cycle\_square}.
The implementation supports the computational use
of the moment constructions developed in this paper.

\section{Conclusion}
\label{sec:conclusion}

This paper develops an algebraic framework for describing polynomial
first-order projective differential moment
invariants. The main result is a one-to-one correspondence established for every
$n\geq 1$,
$$
\Phi_n:
(\mathcal P_{n,-1})^{S_n}
\longrightarrow
\mathcal{MI}_n,
\qquad
\Phi_n^{-1}=\operatorname{Pol}_n,
$$
between moment invariants of degree $n$ and symmetric
$n$-point polynomial relative projective
invariants of weight $-1$. Thus, the problem of constructing
moment invariants reduces to the already described
algebraic problem of constructing multipoint relative
invariants. The injectivity of the integral realization on
the subalgebra of moments arising in these constructions
ensures that the formal algebraic results carry over
to the corresponding image functionals.

Explicit results are obtained for low moment degrees.
The absence of nonzero invariants of degrees
$1$ and $2$ is proved, and a complete description of invariants of degree
$3$ is obtained through symmetrized multipoint kernels. The ten
cubic differential moments form a module
equivalent to the space of ternary cubics, so the classical
Aronhold invariants $S$ and $T$ generate projective
differential moment invariants of degrees $4$ and $6$.
In particular, this connection explains the invariant-theoretic
nature of the known 25-term expression
$\operatorname{Inv}_{2,0}$ \cite{Wang}.

The infinite dimensionality of
the space $\mathcal{MI}_3$ and the absence of
a finite generating system for the algebra
$\mathcal{MI}$ are also established. At the same time, every rational
moment invariant is a ratio of two
polynomial moment invariants.
Thus, the field of rational moment invariants
is completely determined by the algebra $\mathcal{MI}$.

Compact multipoint representations provide
a constructive passage from relative invariants
to moment polynomials and serve as a basis
for their software computation.
This connects the algebraic results obtained
with the construction of invariant image characteristics
for image analysis and pattern recognition tasks.

Further development of this approach involves
designing numerically stable methods
for computing the constructed invariants
on digital images, using
orthogonal moments, and evaluating
the discriminative power of the resulting characteristics
in image analysis and pattern recognition tasks.

\section*{Appendix}


\begin{align}
T={}&
A^2B^2C^2
-3A^2E^2H^2
-6A^2BEHC
+4A^2BH^3
+4A^2E^3C
\nonumber\\
&-6ADGBC^2
+18ADGEHC
-12ADGH^3
+12ADJBHC
-24ADJE^2C
\nonumber\\
&+12ADJEH^2
-12ADFBH^2
+6ADF BEC
+6ADFE^2H
\nonumber\\
&+6AIGBHC
-12AIGE^2C
+6AIGEH^2
+12AIJBEC
+12AIJE^2H
\nonumber\\
&-6AIFB^2C
+18AIFBEH
-24AG^2JHC
-24AIJBH^2
-12AIFE^3
\nonumber\\
&+4AG^3C^2
-12AG^2FEC
+24AG^2FH^2
+36AGJ^2EC
+12AGJ^2H^2
\nonumber\\
&+12AGJFBC
-60AGJFEH
-12AGF^2BH
+24AGF^2E^2
-20AJ^3BC
\nonumber\\
&-12AJ^3EH
+36AJ^2FBH
+12AJ^2FE^2
-24AJF^2BE
+4AF^3B^2
\nonumber\\
&+4D^3BC^2
-12D^3EHC
+8D^3H^3
+24D^2IE^2C
-12D^2IEH^2
\nonumber\\
&+12D^2GJHC
+6D^2GFEC
-24D^2J^2H^2
-12D^2IBHC
-3D^2G^2C^2
\nonumber\\
&-24G^2J^2F^2
+24GJ^4F
-12D^2GFH^2
+12D^2J^2EC
-24D^2J FBC
\nonumber\\
&-27D^2F^2E^2
+36D^2JFEH
+24D^2F^2BH
+24DI^2BH^2
-12DI^2BEC
\nonumber\\
&-12DI^2E^2H
+6DIG^2HC
-60DIGJEC
+36DIGJH^2
+18DIGFBC
\nonumber\\
&-6DIGFEH
+36DIJ^2BC
-12DIJ^2EH
-60DIJFBH
+36DIJFE^2
\nonumber\\
&+6DIF^2BE
+12DG^2JFC
-12DGJ^3C
-12DGJ^2FH
+36DGJF^2E
\nonumber\\
&-12DGF^3B
+24DJ^4H
+12DJ^2F^2B
+4I^3B^2C
+24I^2G^2EC
\nonumber\\
&-27I^2G^2H^2
-36DJ^3FE
-12I^3BEH
+8I^3E^3
-24I^2GJBC
\nonumber\\
&+36I^2GJEH
+6I^2GFBH
+12I^2J^2BH
-3I^2F^2B^2
-12DG^2F^2H
\nonumber\\
&-12I^2GFE^2
-24I^2J^2E^2
+12I^2JFBE
-12IG^3FC
+12IG^2J^2C
\nonumber\\
&+36IG^2JFH
-12IG^2F^2E
-36IGJ^3H
-12IGJ^2FE
+12IGJF^2B
\nonumber\\
&+24IJ^4E
-12IJ^3FB
+8G^3F^3
-8J^6 .
\label{eq:Aronhold_T}
\end{align}


\begin{thebibliography}{20}

\bibitem{Hu}
Hu M. K.
Visual pattern recognition by moment invariants.
\emph{IRE Transactions on Information Theory}
1962, \textbf{8}(2), 179--187.
doi:10.1109/TIT.1962.1057692.

\bibitem{FSB}
Flusser J., Suk T., Zitova B.
\emph{2D and 3D Image Analysis by Moments}.
John Wiley \& Sons, 2017.

\bibitem{SukFlusser2004}
Suk T., Flusser J.
Projective moment invariants.
\emph{IEEE Transactions on Pattern Analysis and Machine Intelligence}
2004, \textbf{26}(10), 1364--1367.
doi:10.1109/TPAMI.2004.89.

\bibitem{WangCoMoments}
Wang Y., Zhang B., Yao T.
Projective invariants of co-moments of 2D images.
\emph{Pattern Recognition}
2010, \textbf{43}(10), 3233--3242.
doi:10.1016/j.patcog.2010.05.004.

\bibitem{Wang}
Wang Y. B., Wang X. W., Zhang B., Wang Y.
Projective invariants of D-moments of 2D grayscale images.
\emph{Journal of Mathematical Imaging and Vision}
2015, \textbf{51}(2), 248--259.
doi:10.1007/s10851-014-0518-z.

\bibitem{Li2019}
Li E., Mo H., Xu D., Li H.
Image projective invariants.
\emph{IEEE Transactions on Pattern Analysis and Machine Intelligence}
2019, \textbf{41}(5), 1144--1157.
doi:10.1109/TPAMI.2018.2832060.

\bibitem{Olver2023}
Olver P. J.
Projective invariants of images.
\emph{European Journal of Applied Mathematics}
2023, \textbf{34}(5), 936--946.
doi:10.1017/S0956792522000298.

\bibitem{Bedratyuk2020}
Bedratyuk L.
2D geometric moment invariants from the point of view of the
classical invariant theory.
\emph{Journal of Mathematical Imaging and Vision}
2020, \textbf{62}(8), 1062--1075.
doi:10.1007/s10851-020-00954-9.

\bibitem{B-2025-1}
Bedratyuk L.
First order joint differential projective invariants.
\emph{Communications in Algebra}
2026, \textbf{54}(5), 1861--1880.
doi:10.1080/00927872.2025.2569449.

\bibitem{BedratyukFrames}
Bedratyuk L.
Joint projective invariants on first jet spaces of point
configurations via moving frames.
\emph{Communications in Algebra}
2026, 1--19.
doi:10.1080/00927872.2026.2638961.

\bibitem{B-2026}
Bedratyuk L.
Polynomial joint first-order differential projective invariants.
arXiv:2608.30016, 2026.

\bibitem{Aronhold1850}
Aronhold S.
Zur Theorie der homogenen Functionen dritten Grades von drei Variabeln.
\emph{J. Reine Angew. Math.}
\textbf{39} (1850), 140--159.
doi:10.1515/crll.1850.39.140.

\bibitem{Olver1993LieGroups}
Olver P. J.
\emph{Applications of Lie Groups to Differential Equations}.
2nd ed., Graduate Texts in Mathematics, Vol.~107.
New York: Springer-Verlag, 1993.
doi:10.1007/978-1-4612-4350-2.

\bibitem{SL-1}
Lie S.
\emph{Vorlesungen \"uber continuierliche Gruppen mit
Geometrishen und anderen Anwendungen}.
Teubner, Leipzig, 1893.

\bibitem{Springer1998}
Springer T. A.
\emph{Linear Algebraic Groups}.
2nd ed.
Boston: Birkh\"auser, 1998.
doi:10.1007/978-0-8176-4840-4.

\bibitem{PopovVinberg1994}
Popov V. L., Vinberg E. B.
Invariant theory.
In: Parshin A. N., Shafarevich I. R., eds.
\emph{Algebraic Geometry IV: Linear Algebraic Groups,
Invariant Theory}.
Encyclopaedia of Mathematical Sciences, Vol.~55.
Berlin: Springer-Verlag, 1994, 123--278.
doi:10.1007/978-3-662-03073-8\_2.

\bibitem{FlusserSuk1993}
Flusser J., Suk T.
Pattern recognition by affine moment invariants.
\emph{Pattern Recognition}
1993, \textbf{26}(1), 167--174.
doi:10.1016/0031-3203(93)90098-H.

\end{thebibliography}
\end{document}